\documentclass{amsart}
\usepackage{amssymb,amscd,amsthm,amsmath}
\usepackage{comment}
\usepackage{fullpage}
\usepackage{epsfig}
\usepackage{graphicx}
\usepackage{xypic}
\usepackage{url}
\usepackage{mathtools}
\usepackage{hyperref}
\usepackage{quiver}
\usepackage{mathrsfs}

\numberwithin{equation}{section}

\makeatletter
\newcommand{\oset}[3][0ex]{%
 \mathrel{\mathop{#3}\limits^{
 \vbox to#1{\kern-2\ex@
 \hbox{\(\scriptstyle#2\)}\vss}}}}
\makeatother

\newcommand{\posmod}{\oset{+}{\rightarrow}}
\newcommand{\negmod}{\oset{-}{\rightarrow}}
\newcommand{\posmodalong}{\oset{+}{\leadsto}}

\newcommand{\biposmod}{\oset{+}{\leftrightarrow}}
\newcommand{\binegmod}{\oset{-}{\leftrightarrow}}

\newcommand{\Stype}[1]{S_{#1}}
\newcommand{\Utype}[3]{U_{#1}(#2,#3)}
\newcommand{\C}{\mathcal{C}}

\newcommand{\PP}{\mathbb{P}}
\newcommand{\OO}{\mathcal{O}}

\newcommand{\NN}{\mathcal{N}}

\newcommand{\Hom}{\textsf{Hom}}

\newcommand{\codim}{\operatorname{codim}}

\newcommand{\rk}{\operatorname{rk}}

\newcommand{\leqor}{\underset{{\scriptscriptstyle (}-{\scriptscriptstyle )}}{<}}

\newcommand{\defi}[1]{\textsf{#1}} 

\newtheorem{theorem}{Theorem}[section]
\newtheorem{lemma}[theorem]{Lemma}
\newtheorem{proposition}[theorem]{Proposition}
\newtheorem{corollary}[theorem]{Corollary}

\theoremstyle{definition}
\newtheorem{definition}[theorem]{Definition}

\newtheorem{remark}[theorem]{Remark}

\definecolor{ao(english)}{rgb}{0.0, 0.5, 0.0}

\input xy
\xyoption{all}

\makeatletter
\def\dasharrowfill@#1#2#3#4{%
 $\m@th
 \thickmuskip0mu
 \medmuskip\thickmuskip
 \thinmuskip\thickmuskip
 \relax
 #4#1\mkern2mu
 \xleaders\hbox{$#4\mkern2mu#2\mkern2mu$}\hfill
 \mkern2mu
 #3$%
}

\def\dashrightarrowfill@{\dasharrowfill@\relbar\relbar\rightarrow}

\providecommand*\xdashrightarrow[2][]{%
  \ext@arrow 0055{\dashrightarrowfill@}{#1}{#2}}
\makeatother

\title{Stability of normal bundles of Brill--Noether curves in \(\mathbb{P}^4\)}

\author{Izzet Coskun}
\address{Department of Mathematics, Statistics, and CS \\
University of Illinois Chicago, Chicago IL 60607}
\email{icoskun@uic.edu}
\author{Eric Jovinelly}
\address{Department of Mathematics \\
Brown University \\
Box 1917 \\
151 Thayer Street \\
Providence, RI, 02912}
\email{eric\_jovinelly@brown.edu}
\author{Eric Larson}
\address{Department of Mathematics \\
Brown University \\
Box 1917 \\
151 Thayer Street \\
Providence, RI, 02912}
\email{elarson3@gmail.com}


\date{}

\begin{document}

\begin{abstract}
We prove that a general Brill--Noether curve \(C\) of genus $g \geq 2$ and degree $d$ in $\PP^4$ has stable normal bundle $N_C$ if and only if $$(g, d) \notin \{(2,6), (3,7), (5,8), (6,9), (7,10)\}.$$ Moreover, \(N_C\) is strictly semistable if $(g, d) \in \{(3, 7), (5, 8)\}$, and is unstable if $(g, d) \in \{(2, 6), (6, 9), (7, 10)\}$.  Our results are valid in any characteristic. Along the way, we also generalize previous results of Larson--Vogt \cite{lv21} on interpolation for \(N_C(-1)\), from characteristic zero to arbitrary characteristic.
\end{abstract}

\maketitle

\section{Introduction}
Let $\rho(g,r,d) = g - (r+1)(g-d + r)$ be the Brill--Noether number.  By the Brill--Noether Theorem, a general curve of genus $g$ admits a nondegenerate map to $\PP^r$ of degree $d$ if and only if $\rho(g,r,d) \geq 0$ \cite{griffithsharris}. Moreover, in this case there is a unique component of the Hilbert scheme whose general point parameterizes nondegenerate (smooth if $r \geq 3$) curves of genus $g$ and degree $d$, and whose moduli map dominates the moduli space of genus $g$ curves $M_g$ \cite{fultonlazarsfeld, gieseker, eisenbudharris87} (for proofs in arbitrary characteristic see \cite{welters, jensenpayne, erichannahisabel}).  A curve in this component is called a \defi{Brill--Noether curve} or \defi{BN-curve} for short.  
In this paper, we characterize the pairs $(g,d)$ for which the general BN-curve of genus $g$ and degree $d$ in $\PP^4$ has (semi)stable normal bundle. We work over an algebraically closed field of arbitrary characteristic.

The normal bundle $N_C= N_{C/\PP^r}$ of a curve $C$ in $\PP^r$ controls the deformations of $C$ in $\PP^r$ and plays a crucial role in applications to arithmetic and moduli. Stability is among the most important properties of a vector bundle. To define it, recall that 
the slope $\mu$ of a vector bundle $E$ on a curve $C$ is defined by $\mu(E)= \frac{\deg(E)}{\rk(E)}.$ The bundle $E$ is \defi{(semi)stable} if every proper subbundle $F$ of $E$ satisfies $$\mu(F) \leqor \mu(E).$$ 
(Semi)stable vector bundles are the building blocks of all vector bundles under filtration, much in the same way that prime (powers) are the building blocks of all integers under multiplication. That is, every vector bundle on a curve admits a unique filtration, called the \defi{Harder--Narasimhan (HN) filtration}, with semistable quotients of decreasing slope,
and every semistable vector bundle admits a filtration, called a \defi{Jordan--H\"{o}lder (JH) filtration}, with stable quotients of the same slope, which are unique up to permutation.

Consequently, the stability of normal bundles is a central question in geometry.  This question was completely settled for general BN-curves in $\PP^3$ by \cite{clv22}.  In higher dimensions, the stability of normal bundles is only known in some cases \cite{br99, cs23, clv23, ran23}.  We refer the reader to \cite[\S 1.3]{cs23} for a more detailed history of the problem and further references. Our main theorem completely settles this fundamental question for general  BN-curves in $\PP^4$.

\begin{theorem}\label{thm-mainstab}
Let $C \subset \PP^4$ be a general BN-curve of genus $g\geq 2$ and degree $d$ over an algebraically closed field of arbitrary characteristic.  The normal bundle $N_C$ is stable if and only if 
$$(g,d) \notin \{(2,6),(3,7),(5,8),(6,9),(7,10)\}.$$
The normal bundle is strictly semistable if $(g,d) \in \{(3,7),(5,8)\}$ and unstable if $(g,d) \in \{(2,6),(6,9),(7,10)\}$.
\end{theorem}

\begin{remark}
For a smooth curve of genus $g$ and degree $d$ in $\PP^r$, a straightforward Chern class computation shows \(\deg(N_C) = (r + 1)d + 2g - 2\). We also have \(\rk(N_C) = r - 1\), so \(\mu(N_C) = ((r + 1)d + 2g - 2) / (r - 1)\).
\end{remark}

\begin{remark}\label{rem-g1}
The normal bundle of a general BN-curve $C$ of genus $1$ in $\PP^r$ is semistable \cite[Theorem~3.1]{cs23}. A semistable bundle on a genus $1$ curve is stable if and only if its degree and rank are coprime \cite{Atiyah}. Consequently, $N_C$ is stable if and only if $2d$ and $r-1$ are relatively prime. Otherwise, $N_C$ is strictly semistable. Hence, when $r=4$, the bundle $N_C$ is stable if and only if  $d \not\equiv 0$ mod $3$; the bundle $N_C$ is strictly semistable when $d \equiv 0$ mod~$3$. 
\end{remark}

\begin{remark}\label{rem-g0}
A vector bundle $E$ on a rational curve is (uniquely) isomorphic to a direct sum of line bundles $E \cong \bigoplus \OO_{\PP^1}(a_i)$. The vector bundle is called \defi{balanced} if $|a_i - a_j| \leq 1$ for all $i, j$. The normal bundle of a general rational curve $C$ of degree $d \geq r$ in $\PP^r$ is balanced except when the characteristic of the ground field is $2$; in characteristic $2$, the normal bundle is as balanced as possible subject to the condition that all $a_i \equiv d$ mod $2$ \cite[Remark~4.4]{CLVgrass}. A balanced vector bundle is semistable if and only if it has integer slope.
Hence, when $r=4$, the bundle $N_{C}$ is never stable, and is strictly semistable if and only if $d \equiv 1$ mod~$3$.
\end{remark}

Our second theorem concerns when curves can pass through certain configurations of points.  
The study of such questions goes all the way back to Euclid, who showed that circles can pass through three given points in the plane --- provided that those points are not collinear.

In the  context of modern Brill--Noether theory, Larson--Vogt \cite{lv22} determined the number of general points a BN-curve can pass through.
It is natural to further ask when  a BN-curve of genus $g$ and degree $d$ can pass through points as some subset of the points specialize onto  a hyperplane. By B\'ezout's theorem, we cannot specialize more than $d$ points onto a hyperplane. Since the normal bundle of a curve controls its deformation theory, this question is intimately linked to a cohomological property known as interpolation for the twist $N_{C}(-1)$.

This question was completely settled for BN-curves in \(\PP^3\) in characteristic zero by \cite{vogt, l21}, and in positive characteristic by \cite{l24}. In \(\PP^4\), an answer to this question was previously  only known in characteristic zero \cite{lv21}, and in even higher dimensions only partial results are known \cite{ballico}.
Along the way to proving our main theorem, we also completely settle this question for BN-curves in \(\PP^4\) in arbitrary characteristic.

\begin{definition}
A vector bundle $E$ on a curve $C$ \defi{satisfies interpolation} if $H^1(E) = 0$ and, for a general effective divisor $D$ of any nonnegative degree, either
$$H^0(E(-D)) = 0 \quad \text{or} \quad H^1(E(-D)) = 0.$$
\end{definition}

\begin{theorem}\label{thm-maininterpolation}
Let $C \subset \PP^4$ be a general BN-curve of degree $d$ and genus $g \geq 1$. Then $N_{C}(-1)$ satisfies interpolation unless 
$$(g,d) \in \{(2,6), (5,8), (6,9), (7,10)\}.$$
\end{theorem}

\begin{remark}\label{rem:char2 exceptions}
When $C$ has genus $0$, the bundle $N_{C}(-1)$ satisfies interpolation if and only if it is balanced. Hence, there are additional exceptions in characteristic $2$ as described in Remark \ref{rem-g0}.  However, these cases are not exceptions to the following corollary concerning a geometric version of interpolation.
\end{remark}

\begin{corollary}\label{cor:geom interpolation}
Let $g$ and $d$ be nonnegative integers satisfying $\rho(g, 4, d) \geq 0$, and let $\Gamma \subset \PP^4$ be a set of $n$ points, which is general subject to the condition that exactly $d$ of them lie on a hyperplane. Then there exists a smooth nondegenerate BN-curve of genus $g$ and degree $d$ passing through $\Gamma$ if and only if
$$n \leq \left\lfloor\frac{5d - g + 1}{3}\right\rfloor,$$
except in the following cases:
$$(g,d) \in \{(5,8), (6,9), (7,10)\}.$$
\end{corollary}

\subsection{The strictly semistable and unstable cases}\label{sec-unstablecases} Here we explain the geometry of the exceptional cases when the normal bundle of a general BN-curve $C$ of genus $g$ and degree $d$ is not stable.

\subsubsection{The case $(g, d)=(5, 8)$} Such curves are canonical curves of genus $5$. When $C$ is not hyperelliptic or trigonal, such a curve is a complete intersection of 3 quadrics in $\PP^4$. Consequently, $N_{C} \cong \OO_C(2)^{\oplus 3}$. Hence, the normal bundle is semistable but not stable, and \(N_C(-1)\) does not satisfy interpolation.
Moreover, since \(8\) general points in \(\PP^3\) do not lie on a net of quadrics, this case is also an exception to Corollary~\ref{cor:geom interpolation}.

\subsubsection{The case $(g,d)= (3,7)$} The general such curve lies on a three-dimensional vector space of quadrics, and is residual to a trisecant line $\ell$ in the intersection of these quadrics. Let $C \cap \ell = \Gamma = \{s_1, s_2, s_3\}$. For each $s_i \in \Gamma$, there exists a unique quadric $Q_i$ containing $C$ and singular at $s_i$. We conclude that $$N_{C} \cong \bigoplus_{i=1}^3 \OO_C(2)(-s_i).$$ Hence, $N_{C}$ is semistable but not stable. 

\subsubsection{The case $(g,d) = (2,6)$\label{ss:62}}  Since $C$ is hyperelliptic, it lies in a surface $S \subset \PP^4$ swept out by lines through the fibers of the $g_2^1$.  The surface $S$ is a cubic scroll and $C$ is the intersection of $S$ with a quadric hypersurface. Hence we have an exact sequence
\[0 \to N_{C/S} \simeq \OO_C(2) \to N_C \to N_S|_C \to 0.\]
Since $\mu(N_{C/S}) = 12$, which is strictly greater than 
$\mu(N_{C}) = 10\frac{2}{3}$,
the subbundle $N_{C/S}$ destabilizes $N_C$. Moreover, twisting down \(N_C(-1)\) by a general divisor $D$ of degree \(4\), we have
\[H^0(N_C(-1)(-D)) \supset H^0(N_{C/S}(-1)(-D)) \neq 0,\] and so \(N_C(-1)\) does not satisfy interpolation.

In fact, we claim \(N_S|_C\) is strictly semistable,
so in particular \(N_{C/S} \subset N_C\) is the HN-filtration.
Indeed, let \(Q\) be the unique quadric containing \(S\) which is singular along its directrix.
Write \(N_{S/Q}|_C \subset N_S|_C\) for the saturation of the corresponding bundle on the smooth locus of $Q$, and \(N_Q|_C\) for the quotient.
Then we have an exact sequence
\[0 \to N_{S/Q}|_C \to N_S|_C \to N_Q|_C \to 0.\]
A straightforward Chern class computation gives \(N_{S/Q}|_C \simeq K_C^\vee(2)\) and \(N_Q|_C \simeq K_C^{\otimes 2}(1)\), both of slope \(10\), so the above sequence gives a JH-filtration of \(N_S|_C\) and in particular \(N_S|_C\) is strictly semistable.

To complete this picture, we claim that \(N_S|_C\) is the (unique) nonsplit extension of \(K_C^{\otimes 2}(1)\) by \(K_C^\vee(2)\). Indeed, to establish this claim, it suffices to show \(\Hom(N_C, K_C^\vee(2)) = 0\). This can be done by degenerating \(C\) to a nodal union \(D \cup L\), where \(D\) is an elliptic curve of degree \(5\) and \(L\) is a \(2\)-secant line meeting \(D\) at \(\{x, y\}\).
Later on, in Lemma~\ref{lem:hartshornehirschowitz}, we will describe \(N_C|_D\) and \(N_C|_L\) as positive modifications of \(N_D\) and \(N_L\);
here it suffices to note that \(N_C|_D\) and \(N_C|_L\) are vector bundles containing
\(N_D\) and \(N_L\) as full-rank subsheaves.
Since \(N_D\) is stable of slope \(8\frac{1}{3}\) while \(K_C^\vee(2)|_D\) is of degree \(8\), every map \(N_D \to K_C^\vee(2)|_D\) is identically zero.
Moreover since \(N_L\) is semistable of slope \(1\) while \(K_C^\vee(2)|_L\) is of degree \(2\), every map \(N_L \to K_C^\vee(2)|_L\) that vanishes at \(\{x, y\}\) is identically zero.
It follows that \(\Hom(N_C, K_C^\vee(2)) = 0\) as desired.

\subsubsection{The case \((g, d) = (6, 9)\)} By Riemann--Roch, $C$ is contained in a pencil of quadrics.
Since \(C\) passes through \(13\) general points by \cite[Theorem 1.2]{lv22}, the pencil is general thus has a smooth base locus. Hence we get a map  
$$0 \to K_C(1) \to N_C \to \OO_C(2)^{\oplus 2} \to 0.$$

We conclude that the HN-filtration of $N_C$ in this case is given by 
$K_C(1) \subset N_C$ with quotient $\OO_C(2)^{\oplus 2}$. Hence, the normal bundle is unstable, and \(N_C(-1)\) does not satisfy interpolation.
Moreover, since \(9\) general points in \(\PP^3\) do not lie on a pencil of quadrics, this case is also an exception to Corollary~\ref{cor:geom interpolation}.

\subsubsection{The case \((g, d) = (7, 10)\)}\label{ss:710} By Riemann--Roch, $C$ is contained in a quadric \(Q\).
Since \(C\) passes through \(14\) general points by \cite[Theorem 1.2]{lv22}, the quadric \(Q\) is general thus smooth.
Hence we get an exact sequence
$$0 \to N_{C/Q} \to N_C \to \OO_C(2) \to 0.$$

In particular, \(N_C\) is unstable and \(N_C(-1)\) does not satisfy interpolation.
Moreover, since \(10\) general points in \(\PP^3\) do not lie on a quadric, this case is also an exception to Corollary~\ref{cor:geom interpolation}.

In fact, we will see that \(N_{C/Q}\) is stable in Lemma~\ref{lem: stability in quadric}, and that \(N_{C/Q}(-1)\) satisfies interpolation in Theorem~\ref{thm:LV base cases}; in particular, the above exact sequence is the HN-filtration.

\subsection*{Overview of the proof}
Section \ref{sec: preliminaries} reviews some preliminaries on the stability of vector bundles on nodal curves and modifications of vector bundles. Section \ref{sec: case reduction} reduces the proofs of our theorems to a finite number of cases using a degeneration originally explored in \cite{lv21}.  Section \ref{sec: interpolation proof} outlines the minor modifications to the argument of \cite{lv21} that are needed to prove interpolation for the twisted normal bundle in arbitrary characteristic.
All told, this suffices to reduce the proof of Theorem~\ref{thm-mainstab} to \(24\) low-degree cases.
As in most previous work on the stability of normal bundles, these low-degree cases are by far the most interesting and subtle ones --- the low-degree cases are, after all, where the exceptions lie --- and prior techniques are inadequate to handle them.
For this, we introduce two fundamental innovations, both of which hold promise for generalizing to higher dimensional projective spaces.

The first, in Section \ref{sec: degeneration arguments}, is a detailed analysis of the behavior of (semi)stability under a degeneration to the union of a curve \(C\) with a bisecant line \(L\). Strikingly, we are able to deduce (semi)stability for the normal bundle of a smoothing of \(C \cup L\) by showing that several different modifications of \(N_C\) are close to (semi)stable.  This is a crucial difference; the bounds we need to prove for such modifications of \(N_C\) are weaker than the bounds required for (semi)stability.  Applying this construction iteratively thereby reduces the problem of (semi)stability to verifying bounds on the slopes of subbundles which get progressively weaker as the degree gets larger. 

The second, in Section \ref{sect: types of bundles}, is a method for showing that making such modifications to bundles on elliptic curves which are ``sufficiently unstable'' produces bundles which are ``less unstable.'' This allows us to efficiently verify the desired slope bounds after pulling off bisecant lines to obtain elliptic curves.

Finally, in Section \ref{sec:final}, we apply these two tools to verify the majority of the  \(24\) low-degree cases, and handle the  remaining ones by ad hoc methods.

\subsection*{Acknowledgments} We would like to thank Eric Riedl, Geoffrey Smith, and Isabel Vogt for invaluable conversations about the stability of normal bundles on curves.

We would also like to acknowledge the generous support of the National Science Foundation, the Simons Foundation, and the Sloan Foundation:
Izzet Coskun was supported by NSF grant DMS-2200684 and Simons Foundation Travel Support for Mathematicians SFI-MPS-TSM-00013580;
Eric Jovinelly was supported by the NSF postdoctoral research fellowship DMS-2303335; Eric Larson was supported by NSF grant DMS-2440719 and Sloan Foundation research fellowship FG-2025-24482.

\section{Preliminaries}\label{sec: preliminaries}

Throughout we work over an algebraically closed field of arbitrary characteristic.  Let $C$ be a smooth curve and $E$ be a vector bundle on $C$.  The \defi{HN-filtration} of $E$ is the unique filtration
\[ 0 = E_0 \subset E_1 \subset \cdots \subset E_{n-1} \subset E_n = E\]
by subbundles such that each $F_i:= E_i/E_{i-1}$ is semistable and
$\mu(F_1) > \mu(F_2) > \cdots > \mu(F_n).$
We call $F_1$ the \defi{maximal destabilizing subbundle} and $F_n$ the \defi{minimal destabilizing quotient}.

\subsection{Stability of Vector Bundles on Nodal Curves} For a connected nodal curve $C$, let $\nu : \tilde{C} \rightarrow C$ be the normalization.  For any node $p \in C$, let $\tilde{p}_1, \tilde{p}_2 \in \tilde{C}$ be the two points lying over $p$.  For any vector bundle $E$ on $C$, there is a natural isomorphism between the fibers of $\nu^* E$ over $\tilde{p}_1$ and $\tilde{p}_2$.  

\begin{definition}
Let $E$ be a vector bundle on a connected nodal curve $C$.  The \defi{adjusted slope} of a subbundle $F \subset \nu^*E$ is 
$$\mu_C^{\text{adj}}(F) := \mu(F) - \frac{1}{\rk F} \sum_{p \in C_{\text{sing}}} \codim_F(F|_{\tilde{p}_1} \cap F|_{\tilde{p}_2}),$$
where $\codim_F(F|_{\tilde{p}_1} \cap F|_{\tilde{p}_2})$ is the codimension of the intersection in either $F|_{\tilde{p}_1}$ or $F|_{\tilde{p}_2}$.  When $C$ is unambiguous, we write $\mu^{\text{adj}}(F)$ for $\mu_C^{\text{adj}}(F)$.  We say that $E$ is (\defi{semi})\defi{stable} if for all nonzero subbundles $F \subsetneq \nu^* E$, 
$$\mu^{\text{adj}}(F) \leqor \mu(\nu^* E) = \mu(E).$$
\end{definition}

\noindent
This definition of (semi)stability behaves well in families of vector bundles on nodal curves.

\begin{proposition}\cite[Proposition 2.3]{clv22}\label{prop: openness of adjusted stability}
Let $\mathcal{C} \rightarrow \Delta$ be a family of connected nodal curves over the spectrum of a discrete valuation ring.  For a vector bundle $\mathcal{E}$ on $\mathcal{C}$, if the special fiber of $\mathcal{E}$ is (semi)stable, then the general fiber of $\mathcal{E}$ is also (semi)stable.
\end{proposition}

\noindent
Moreover, we may infer (semi)stability on the total curve $C$ from (semi)stability on each of the components.

\begin{lemma}\cite[Lemma 4.1]{clv22}\label{lem: reducible stability}
Suppose $C = X \cup Y$ is a reducible nodal curve.  If $E$ is a vector bundle on $C$ such that $E|_X$ and $E|_Y$ are semistable, then $E$ is semistable.  Furthermore, if in addition one of $E|_X$ or $E|_Y$ is stable, then $E$ is stable.    
\end{lemma}

\subsection{Modifications of Vector Bundles} 
The normal bundle $N_C$ of a reducible nodal curve $C \subset \PP^r$ may be expressed in terms of the normal bundles of its components. To explain this relationship, we will use elementary modifications of vector bundles.

\begin{definition}
Let $X$ be a scheme and $E$ be a vector bundle on $X$.  Let $D \subset X$ be a Cartier divisor and $F \subset E|_D$ be a subbundle of the restriction of $E$ to $D$.  The \defi{negative elementary modification} of $E$ along $D$ toward $F$ is
$$E[D \negmod  F] := \ker(E \rightarrow E|_D/F).$$
The \defi{positive elementary modification} of $E$ along $D$ toward $F$ is the twist $E[D \posmod  F] := E[D \negmod  F](D)$.
\end{definition}

\begin{remark}\label{rem:sequence on D}
Given a vector bundle $E$, a Cartier divisor $D$ and a subbundle $F \subset E|_D$, there are exact sequences
\begin{align*} 0 \to E|_D/F  \otimes \OO(-D)|_D \to E[D \negmod F]|_D &\to F \to 0; \\
0 \to E|_D/F \to E[D \posmod F]|_D &\to F \otimes \OO(D)|_D \to 0.
\end{align*}
\end{remark}

Observe that the modifications $E[D \posmod  F]$ are naturally isomorphic to $E$ over the complement of $D$ in $X$.  Thus, given disjoint divisors $D_1, D_2 \subset X$ and a subbundle $F_i \subset E|_{D_i}$ for each $i$, we can define the iterated modification $E[D_1 \posmod  F_1][D_2 \posmod  F_2]$ without ambiguity.  If instead the supports of $D_1$ and $D_2$ overlap, the subbundles $F_i$ may not be enough information to define an iterated modification.  For this reason, when considering multiple modifications with overlapping supports, we will assume each subbundle $F_i \subset E|_{D_i}$ extends to a subbundle of $E$ on an open neighborhood $U_i$ of $D_i$.  If the subbundle $F_2 \subset E|_{U_2 \setminus D_1}$ extends to a subbundle of $E[D_1 \posmod  F_1]$ over some neighborhood of $D_2$, then it does so uniquely and we may use this extension to define $E[D_1 \posmod  F_1][D_2 \posmod  F_2]$.  However, such an extension may not always exist.  The following condition guarantees such an extension exists and furthermore guarantees multiple modifications satisfy commutativity, i.e., $E[D_1 \posmod  F_1][D_2 \posmod  F_2] = E[D_2 \posmod  F_2][D_1 \posmod  F_1]$, along with other nice properties.

\begin{definition}\label{def:tree_like}
Let \(M = \{(D_i, U_i, F_i)\}_{i \in I}\) be a collection of modification data.
For each point \(x \in X\), define \(I_x \subseteq I\) to be the set of indices for which \(x \in D_i\).
We say that \(M\) is \defi{tree-like} if for all \(x \in X\), and all subsets \(I' \subset I_x\), the following condition holds:
whenever the fibers \(\{F_i|_x\}_{i \in I'}\) are dependent, there exist indices \(i,j \in I'\)
and an open neighborhood \(U\) of \(x\) such that \(F_i|_U \subseteq F_j|_U\).
\end{definition}

We refer to \cite[Section~2]{aly} and \cite[Section~3]{lv22} for more details.  All modifications in this paper will be tree-like.   Furthermore, most modifications will be towards pointing bundles, defined below.

\begin{definition}\label{def:pointingbundles}
Let $C \subset \PP^r$ be a nodal curve and $\Lambda \subset \PP^r$ be a linear subscheme that is disjoint from the singular locus of $C$.  Let $\pi : \PP^r \dashrightarrow \PP^{r-\lambda}$ be the projection away from $\Lambda$.  Suppose $\pi|_C : C \rightarrow \overline{C} \subset \PP^{r - \lambda}$ is unramified.  The \defi{pointing bundle} toward $\Lambda$ is 
$$N_{C \rightarrow \Lambda} :=\ker\big (\pi_* : N_C \rightarrow N_{\overline{C}}(\Lambda \cap C)\big ).$$
Informally, sections of $N_{C \rightarrow \Lambda}$ point towards $\Lambda$.  For any subscheme $\Gamma \subset \PP^r$, we set $N_{C \rightarrow \Gamma}$ to be the pointing bundle toward the linear span of $\Gamma$.
\end{definition}

For a subscheme $\Gamma \subset \PP^r$, we write $N_C[x \posmod  \Gamma]$ for the positive elementary modification of $N_C$ along $x$ toward $N_{C \to \Gamma}$.  Suppose $\Gamma_1, \Gamma_2 \subset \PP^r$ are two disjoint linear subspaces and $\Gamma_1 + \Gamma_2$ represents their span.  We will use the following fact several times: if $T_x C$ is disjoint from $\Gamma_1 + \Gamma_2$, then \[N_C[x \posmod \Gamma_1][x \posmod \Gamma_2] = N_C[x \posmod \Gamma_1 + \Gamma_2].\]  If in addition $\Gamma_1 + \Gamma_2$ is a codimension two linear subspace, then $ N_C[x \posmod \Gamma_1 + \Gamma_2] \cong N_C(x)$.

\medskip
\noindent
If \(x, y \in C\), then for brevity we write
\[N_C[x \biposmod y] = N_C[x \posmod y][y \posmod x] \quad \text{and} \quad N_C [x \binegmod y] = N_C[x \negmod y][y \negmod x].\]
We most frequently use pointing bundles towards a point $p$.  If $p$ is a general point in $\PP^r$, then $N_{C\rightarrow p} \cong \OO_C(1)$ (\cite[Proposition 6.2]{aly}); if instead $p$ is a general point in $C$, then $N_{C \rightarrow p} \cong \OO_C(1)(2p)$ (\cite[Proposition 6.3]{aly}).

\begin{remark}
Pointing bundles allow one to make inductive arguments regarding normal bundles of general nonspecial BN-curves in $\PP^r$.  Indeed, let $C \subset \PP^r$ be a general BN-curve of genus $g$ and degree $d$ such that $d \geq g + r$.
Then \(H^1(C, \OO_C(1)) = 0\), and so \(C\) can be described as a general curve of genus \(g\), equipped with a general line bundle \(L\) of degree \(d\), and a general subspace \(V \subset H^0(C,L)\) of dimension \(r + 1\). Hence, if \(\Lambda = p\) is a general point in $\PP^r$ (resp.\ in $C$), then $\overline{C} \subset \PP^{r - 1}$ is similarly a general curve of genus \(g\), equipped with a general line bundle \(L\) of degree \(d\) (resp.\ of degree \(d - 1\)), and a general subspace \(V \subset H^0(C,L)\) of dimension \(r\). In particular \(\overline{C} \subset \PP^{r-1}\) is a general BN-curve.  This allows us to make use of previous results about normal bundles of general BN-curves in $\PP^3$.
\end{remark}

\begin{remark}\label{rem: exact sequence mod}
By definition, $N_{C \to \Lambda}$ fits into an exact sequence
$$0 \to N_{C \to \Lambda} \to N_C \to N_{\overline{C}}(\Lambda \cap C) \to 0.$$
Elementary modifications of $N_C$ towards subschemes of $\PP^r$ also modify this exact sequence.  More specifically, if $E$ is a modification of $N_C$ and $N_{C\to \Lambda}'$ is the saturation of $N_{C \to \Lambda}$ in $E$, then there is an exact sequence
$$0 \to N_{C \to \Lambda}' \to E \to \overline{E} \to 0$$
where $\overline{E}$ is a modification of $N_{\overline{C}}(\Lambda \cap C)$.  Informally, modifications that point towards $\Lambda$ modify the saturation of $N_{C \to \Lambda}$ whereas all other modifications modify $N_{\overline{C}}(\Lambda \cap C)$.
\end{remark}

We now introduce an additional notation for elementary modifications.
Let \(C = X \cup Y\) be a nodal curve and \(q \in X \cap Y\).
Then \(T_Y|_q\) defines a subspace of \(N_X|_q\), and we write $N_X[q \posmodalong Y]$ for the modification $N_X[q \posmod  T_Y|_q]$ towards this subspace. For any choice of point \(q' \in T_q Y \setminus \{q\} \subset \PP^r\), this coincides with the modification $N_X[q \posmod  q']$ towards the point $q'$. If $\{q_1, \ldots , q_n\} = X \cap Y$, we write $N_X[\posmodalong Y]$ for the iterated modification $N_X[q_1 \posmodalong Y]\cdots [q_n \posmodalong Y]$.  We have the following fundamental observation.

\begin{lemma}\cite[Corollary 3.2]{hartshornehirschowitz} \label{lem:hartshornehirschowitz}
Let $C = X \cup Y$ be a nodal curve with $X \cap Y = \{q_1, \ldots , q_n\}$.  Then
$$N_C|_X \cong N_X[\posmodalong Y].$$
\end{lemma}

\subsection{Reducible Brill--Noether Curves} 

The most common examples of reducible BN-curves we will consider are unions of irreducible BN-curves with bisecant lines and five-secant twisted cubics.  The following verifies that such curves are BN-curves.

\begin{lemma}\cite[Lemma 2.5]{cs23}\label{lem: reducible BN-curves}
Let $C \subset \PP^r$ be a  BN-curve of genus $g$ and degree $d$.  Consider a rational normal curve $R$ of degree $e$ which intersects $C$ quasitransversely along a set of points $\Gamma \subset C$.  The nodal union $C \cup_\Gamma R$ is a BN-curve if $|\Gamma| \leq e + 1$; if $|\Gamma| = e + 2$ and $e = r$; or if $|\Gamma| = e + 2$, $e = r -1$, $C$ is transverse to the span of $\Gamma$, and $\rho(g,r,d) \geq 1$.
\end{lemma}

Much of our strategy involves further specializing the reducible curves $C \cup R$ in Lemma \ref{lem: reducible BN-curves} and studying what happens to $N_C[\posmodalong R]$ under this specialization.  The following lemma explains one of the more intricate such specializations.

\begin{lemma}\label{lem: onion degeneration}
Let $C \subset \PP^r$ be a BN-curve of degree at least $r+1$ and $R \subset \PP^r$ be an $(r+1)$-secant rational normal curve of degree $r-1$.  Let $p \in C$ be a general point, $H$ be a general hyperplane that is tangent to $C$ at $p$, and $q_1, \ldots, q_{r-1} \in C \cap H$ be any $r-1$ points distinct from \(p\) and each other.  Let $L_i$ be the line spanned by $p$ and $q_i$ for all $1 \leq i \leq r-1$.  
    
As an $(r+1)$-secant rational normal curve, $R$ can specialize to $L_1 \cup \cdots \cup L_{r-1}$.  Moreover, this specialization may be chosen so that $N_C[\posmodalong R]$ specializes to $N_C[q_1 \posmod p] \ldots [q_{r-1} \posmod p][p \posmod V]$, where $V \subset N_C|_p$ is a linear subspace of dimension \(2\). Furthermore, as \(H\) varies among tangent hyperplanes to $C$ at $p$, the subspace \(V\) becomes linearly general.
\end{lemma}
\begin{proof}
The existence of the specialization in question is proven in \cite[Section~7]{lv22}. Linear generality of $V$ is proven in \cite[Lemma~7.5(1)]{lv22}.
\end{proof}

\section{Reduction to Finitely Many Cases}\label{sec: case reduction}

In this section, we use results about normal bundles of BN space curves to reduce the verification of Theorems \ref{thm-mainstab} and \ref{thm-maininterpolation} to a finite list of cases.

\subsection{Attaching curves in hyperplanes}

Let $C \subset \PP^4$ be a general BN-curve of genus $g$ and degree $d$. Let $H \subset \PP^4$ be a hyperplane meeting $C$ transversely.  Assume $N_C(-1)$ satisfies interpolation so that $C \cap H$ is general. Suppose
$$(g',d') \in \{(0,3), (1,4), (3,6), (4,7), (5,8), (6,9)\}$$ and that $d \geq s := \frac{d' + g' -1}{2}$. Observe that $s \leq 2d'$ in these \(6\) cases. Thus we can find a BN space curve $C' \subset H$ of genus $g'$ and degree $d'$ through $s$ general points in $H$ \cite[Theorem 1.2]{lv22}, in particular, one passing through a set $\Gamma$ consisting of $s$ of the intersection points of $C$ with $H$. 

\begin{definition}
A nodal curve is \defi{strongly smoothable} if there exists a smoothing family whose total space is smooth.
\end{definition} 

\begin{lemma}\label{lem:smoothing}
The nodal curve $D= C \cup_\Gamma C'$ is a strongly smoothable BN-curve of genus $g + g' + \frac{d' + g' -3}{2}$ and degree $d + d'$, under the additional assumption that $\rho(g, 4, d) \geq 3$
    if $(g', d') = (6,9)$.
\end{lemma}

\begin{proof}
If \((g', d) \neq (6, 9)\), then $D$ is a BN-curve by \cite[Theorem 1.9]{l22}.  Although \cite{l22} assumes characteristic zero, the proof of Theorem~1.9 goes through in arbitrary characteristic. Indeed, tracing through the proof, the characteristic assumption enters in three ways: (1) as a hypothesis for \cite[Corollary 1.4]{aly}; (2) to go from dominance to generic smoothness for certain maps between moduli spaces; (3) to verify that general tangent lines were in linear general position. But using \cite[Theorem 1.4]{lv22} in place of \cite[Corollary 1.4]{aly} eliminates (1) and renders (2) unnecessary. Moreover, by Lemma~\ref{lem: reducible BN-curves} a general BN-curve of any genus and degree can be specialized to a reducible curve, one of whose components is a nondegenerate rational curve, thereby establishing (3).

The case $(g',d')=(6,9)$ follows as in the proof of \cite[Lemma 3.6]{lv21}, whose only invocation of the characteristic zero hypothesis is \cite[Theorem 1.9]{l22}.

To see that these curves are strongly smoothable, let $N_{D}' \subset N_D$ be the subsheaf of the normal bundle that does not smooth the nodes. We then have  the exact sequence
$$0 \to N_D|_C(-\Gamma) \to N_D' \to N_{C'} \to 0.$$ Since the outer terms in this sequence have vanishing $H^1$, we conclude that $H^1(D,N_D')=0$. Hence, $D$ is strongly smoothable by \cite[Proposition~1.1]{hartshornehirschowitz}. 
\end{proof}

\begin{proposition}\label{prop: smoothing P3}
 Let $C''$ be a general smoothing of the curve $D$ in Lemma \ref{lem:smoothing}.  
\begin{enumerate}
\item\label{prop:smp31} If $N_{C}(-1)$ satisfies interpolation, then $N_{C''}(-1)$ satisfies interpolation.
\item\label{prop:smp32} If $N_{C}$ is (semi)stable, then $N_{C''}$ is (semi)stable.
\end{enumerate}
\end{proposition}

\begin{proof}
Let $\pi : \mathcal{C} \rightarrow B$ be a one-parameter family of curves  smoothing $D$, with smooth total space.  The normal bundle of each fiber in $\PP^4$ forms a bundle $\NN$ on $\mathcal{C}$.  By the second exact sequence in Remark \ref{rem:sequence on D}, the restriction to $C'$ of the modification $\NN(-1)[C' \posmod  N_{C'/H}]$ fits into an exact sequence
$$0 \rightarrow \OO_{C'}(\Gamma) \rightarrow \NN(-1)[C' \posmod  N_{C'/H}]|_{C'} \rightarrow N_{C'/H}(-1)(-\Gamma) \rightarrow 0.$$
Since \(s \geq g'\), we have \(H^1(C', \OO_{C'}(\Gamma)) = 0\).
Because $N_{C'/H}(-1)(-\Gamma)$ and $\OO_{C'}(\Gamma)$ are both (semi)stable \cite{clv22}, satisfy interpolation \cite{lv22}, and 
$$\mu(N_{C'/H}(-1)(-\Gamma)) =  d' + g' - 1 - \frac{d' + g' -1}{2} = \frac{d' + g' -1}{2}=\mu(\OO_{C'}(\Gamma)),$$
$\NN(-1)[C' \posmod  N_{C'/H}]|_{C'}$ is semistable and satisfies interpolation as well.  Similarly,
$$\NN(-1)[C' \posmod  N_{C'/H}]|_C(-\Gamma) \cong N_C(-1)$$ is a (semi)stable bundle which satisfies interpolation.  This proves \eqref{prop:smp32} by Proposition \ref{prop: openness of adjusted stability} and Lemma \ref{lem: reducible stability}.

As the restriction $\NN(-1)[C' \posmod  N_{C'/H}]|_{C'}$ satisfies interpolation and has integral slope, it has a twist down with \(H^0 = H^1 = 0\). Since $\NN(-1)[C' \posmod  N_{C'/H}]|_C(-\Gamma)$ satisfies interpolation, \cite[Lemma~2.10]{lv21} implies $\NN(-1)[C' \posmod  N_{C'/H}]|_D$ satisfies interpolation as well.  This proves \eqref{prop:smp31} by upper semicontinuity.
\end{proof}

\begin{theorem}\label{thm: basecases}
Theorems \ref{thm-mainstab} and \ref{thm-maininterpolation} can be reduced to a finite number of cases as follows.
\begin{enumerate}
\item Assume that for the general BN-curve \(C \subset \PP^4\) of genus $g$ and degree $d$ with
$$(g,d) \in \{(5,9), (6,11), (8,11), (8,12), (10,12), (10,13), (11,13), (12,14)\},$$ the bundle $N_C(-1)$ satisfies interpolation. Then Theorem \ref{thm-maininterpolation} holds.

\item Assume that the normal bundle of the general BN-curve in $\PP^4$ of genus $g$ and degree $d$ with
$$(g,d) \in \{(2,7), (2,8), (2,9), (3,8), (3,10), (4,8), (4,9), (4,10), (5,9), (5,10), (5,11), (6,10), $$
$$ (6,11), (7,12), (8,11), (8,12), (9,13), (10,12), (10,13), (11,13), (11,14), (12,14), (13,15), (15,16)\}$$
is stable. Then Theorem \ref{thm-mainstab} holds.
\end{enumerate}
\end{theorem}

\begin{proof}
We argue by induction on $d$.  Consider any case with $g \geq 16$. Then $\rho(g-12, 4, d-9) = \rho(g,4,d) + 3 \geq 3$ and $g-12 \geq 4$. 
The only exceptions to  Theorems \ref{thm-mainstab} and \ref{thm-maininterpolation} either have $\rho(g, 4, d) < 3$ or $g < 4$. Therefore, if $C$ is a general BN-curve of degree $d-9$ and genus $g-12$, then the conclusions of the theorems hold by induction.
The inequalities $\rho(g,4,d) 
\geq 0$ and $g \geq 16$ imply that $d\geq 17$. In particular, $d-9 \geq 7$. Therefore,  we may  apply Lemma \ref{lem:smoothing} and Proposition \ref{prop: smoothing P3} with $(g',d') = (6,9)$ to conclude that the theorems hold for curves of genus $g$ and degree $d$.

We may thus assume that $g \leq 15$. For each such $g$, let $d_{\min}(g)$ be the minimum value of $d$ for which $\rho(g,4,d) \geq 0$ and $(g,d) \notin \{(2,6),(3,7),(5,8),(6,9),(7,10)\}$. Observe that both theorems have no exceptions for $d \geq d_{\min}$. Consider any case with $d \geq d_{\min} + 3$.  If $C$ is a general BN-curve of degree $d-3$ and genus $g$, then the conclusions of the theorems hold by induction. Therefore, we may apply Lemma \ref{lem:smoothing} and Proposition~\ref{prop: smoothing P3} with $(g',d') = (0,3)$ to conclude that the theorems hold for curves of degree $d$ and genus $g$.
It thus suffices to consider the finitely many cases with $g \leq 15$ and $d \leq d_{\min} + 2$.

Suppose \(g \leq 5\).  Recall that the normal bundle of a general elliptic curve in $\PP^4$ of degree $5$ or $7$ is stable by Remark \ref{rem-g1}.  The only $(g,d)$ with $d \leq d_{\min} + 2$ that is not listed as an exception or base case for Theorem \ref{thm-mainstab} is $(g,d) = (3,9)$.  This case follows from an application of Lemma \ref{lem:smoothing} and Proposition \ref{prop: smoothing P3} with $(g',d') = (1,4)$.  For Theorem \ref{thm-maininterpolation}: when $d \geq 2g$, \cite[Corollary~3.5]{lv21} and \cite[Theorem~1.4]{lv22} prove $N_C(-1)$ satisfies interpolation apart from the listed exceptions.  Again apart from these exceptions, the only $(g,d)$ with $g \leq 5$ and $d < 2g$ is $(g,d) = (5,9)$, in which case Theorem \ref{thm-maininterpolation} holds by assumption.

When $6 \leq g \leq 15$, the following table explains how to further reduce the cases to those listed in the theorem.  Below, $Y =(g_Y,d_Y)$ is a general BN-curve in $\PP^4$ of genus $g_Y$ and degree $d_Y$; $X = (g_X,d_X)$ is a BN space curve of genus $g_X$ and degree $d_X$ that meets $Y$ quasi-transversely at $\frac{d_X + g_X - 1}{2}$ points; and we apply Lemma \ref{lem:smoothing} and Proposition \ref{prop: smoothing P3} with $C = Y$ and $C' = X$ to verify Theorem \ref{thm-maininterpolation} for the indicated pair \((g, d)\).
\begin{center}
\begin{tabular}{c | c | c | c | c}
$g$ & $d_{\min}$ & certificate for $(g, d_{\min})$ & certificate for $(g, d_{\min}+1)$ & certificate for $(g, d_{\min}+2)$ \\ \hline
6 & 10 & $X = (3,6)$, $Y= (0,4)$ & by assumption & $X=(1,4)$, $Y=(4,8)$  \\ 
7 & 11 & $X = (3,6)$, $Y= (1,5)$ & $X = (3,6)$, $Y = (1,6)$ & $X = (3,6)$, $Y = (1,7)$ \\
8 & 11 & by assumption & by assumption & $X = (3,6)$, $Y = (2,7)$ \\
9 & 12 & $X = (4,7)$, $Y= (1,5)$ & $X = (3,6)$, $Y = (3,7)$ & $X = (3,6)$, $Y = (3,8)$ \\
10 & 12 & by assumption & by assumption & $X = (3,6)$, $Y = (4,8)$ \\
11 & 13 & by assumption & $X = (4,7)$, $Y = (3,7)$ & $X = (3,6)$, $Y = (5,9)$ \\
12 & 14 & by assumption & $X = (4,7)$, $Y = (4,8)$ & $X = (3,6)$, $Y = (6,10)$ \\
13 & 15 & $X = (5,8)$, $Y = (3,7)$ & $X = (4,7)$, $Y = (5,9)$ & $X = (3,6)$, $Y = (7,11)$ \\
14 & 16 & $X = (5,8)$, $Y = (4,8)$ & $X = (3,6)$, $Y = (8,11)$ & $X = (3,6)$, $Y = (8,12)$\\
15 & 16 & $X = (6,9)$, $Y = (3,7)$ & $X = (6,9)$, $Y = (3,8)$ & $X = (6,9)$, $Y = (3,9)$ 
\end{tabular}
\end{center}
Observe that $N_Y$ is stable by assumption for all choices of $Y$ in the table aside from $Y \in \{(0,4),(1,6), (3,7)\}$.  Thus, unless $(g,d) \in \{(6,10),(7,12), (9,13),(11,14),(13,15),(15,16)\}$, this argument also proves Theorem \ref{thm-mainstab} for the indicated pairs $(g,d)$ by Proposition \ref{prop: smoothing P3}(2).  Hence, these are the additional base cases needed for Theorem \ref{thm-mainstab} with $g > 5$.
\end{proof}

\section{Interpolation for the Twisted Normal Bundle}\label{sec: interpolation proof}

\noindent
In this section we prove Theorem \ref{thm-maininterpolation} and Corollary \ref{cor:geom interpolation}.  We will need the following lemma. 

\begin{lemma}\cite[Lemma 3.4]{cs23} \label{lem: pull off 1-secant}
Let \(\C \subset \PP^r \times B \to B\) be a one-parameter family of embedded curves with smooth total space. Suppose that the central fiber \(\mathcal{C}_0\) is the union \(C \cup_u L\) of a curve \(C\) with a \(1\)-secant line \(L\) meeting \(C\) at a point \(u \in C\).

Let \(\mathcal{N}\) be a vector bundle on \(\mathcal{C}\) equipped with an isomorphism to \(N_{\mathcal{C} / \PP^r \times B}\) over an open subset of \(\mathcal{C}\) containing \(L\), and let  \(v \neq u \in L\) be another point. Let \(N \subset \mathcal{N}|_C\) be the subsheaf which does not smooth the node at $u$, so that $N$ is naturally isomorphic to $N_C$ in a neighborhood of $u$, and let $\mathcal{C}_b$ be a general fiber of $\mathcal{C} \to B$.  Then there is a flat family whose general fiber is $\mathcal{N}|_{\mathcal{C}_b}$ and whose special fiber is \(N[2u \posmod v](u)\).  In particular, if \(N[2u \posmod v](u)\) has no subbundles of rank \(k \in \mathbb{N}\) and slope at least \(\mu \in \mathbb{R}\), then neither does the restriction $\mathcal{N}|_{\mathcal{C}_b}$.  Similarly, if \(N[2u \posmod v](-1)\) satisfies interpolation, then so does the restriction \(\mathcal{N}(-1)|_{\mathcal{C}_b}\).  
\end{lemma}

\begin{remark}\label{rem: thanksgeoffandizzet}
In fact, if \(N[2u \posmod v](u)\) has no subbundles of rank \(k \in \mathbb{N}\) and slope at least \(\mu \in \mathbb{R}\), then neither does $\mathcal{N}|_{\mathcal{C}_0}$ (using adjusted slopes). This is a slight generalization of \cite[Lemma 4.3]{cs23} which can be proven in the same way.
\end{remark}

\begin{proof}
For the sake of completeness we recall the proof here.
By smoothness of $\mathcal{C}$, we may contract the component $L \subset \mathcal{C}_0$ and obtain a smooth surface $\mathcal{C}^- \rightarrow B$.  We will construct a vector bundle on $\mathcal{C}^-$ whose restrictions to $C$ and $\mathcal{C}_b$ are the indicated bundles.  The result then follows from openness of the desired properties.

The restriction of $\mathcal{N}$ to $L \subset \mathcal{C}_0$ is isomorphic to $\mathcal{O}_L(2) \oplus \mathcal{O}_L(1)^{\oplus r - 2}$.  The modification $\mathcal{N}[L \posmod  \mathcal{O}_L(2)](L)$ restricts to a trivial vector bundle on $L$ and to $\mathcal{N}|_{C}[2u \posmod v](u)$ on $C$ by \cite[Section~2.2]{clv23}.  Hence, $\mathcal{N}[L \posmod  \mathcal{O}_L(2)](L)$ descends to a vector bundle $\overline{\mathcal{N}}$ on $\mathcal{C}^-$, which is the desired flat family.  
    
To see the final claim about interpolation, since $\mathcal{O}(1)(L)$ restricts trivially to \(L\), it descends to a line bundle on \(\mathcal{C}^-\). Twisting down by this line bundle yields the desired claim.
\end{proof}

\noindent
Recall that by Theorem \ref{thm: basecases}, to prove Theorem \ref{thm-maininterpolation} we only need to consider the cases $$(g,d) \in \{ (5,9), (6,11), (8,11), (8,12), (10,12), (10,13), (11,13), (12,14) \}.$$  A majority of these cases are proven by minor changes to \cite{lv21}.

\begin{theorem}\cite{lv21}\label{thm:LV base cases}
For a general BN-curve \(C\) of genus $g$ and degree $d$ in $\PP^4$ with
$$(g,d) \in \{(5,9), (8,11), (10,12), (10,13), (11,13), (12,14)\},$$ 
$N_C(-1)$ satisfies interpolation.

Moreover, for a general BN-curve \(C\) of genus \(7\) and degree \(10\) lying on a smooth quadric \(Q \subset \PP^4\), the twisted normal bundle \(N_{C/Q}(-1)\) satisfies interpolation.
\end{theorem}
\begin{proof}
In characteristic zero, \cite{lv21} proves interpolation for such pairs.  The proof for 
$$(g,d) \in \{(10,12), (10,13), (11,13), (12,14) \}$$
is contained in \cite[Sections~4.1--4.4]{lv21}; the proof for $(g,d) = (8,11)$ lies in \cite[Section~5.1]{lv21}; the proof for $(g,d) = (5,9)$ lies in \cite[Section~6.1]{lv21}; and the proof for \((g, d) = (7, 10)\) lies in \cite[Section~6.6]{lv21}.  While not all results in \cite{lv21} are characteristic independent, these results either already are or become so with slight changes.  We explain these changes below.

Everything referenced from \cite{l22} in these proofs is characteristic independent.  Besides \cite[Theorem~1.3]{aly} and \cite[Appendix~B]{aly}, everything referenced from \cite{aly} is characteristic independent as well.  We can replace \cite[Theorem~1.3]{aly} with \cite[Theorem~1.4]{lv22}.  All other arguments in the proofs of the indicated cases are characteristic independent, so it suffices to consider the use of \cite[Appendix~B]{aly}. This is used twice, for modifications of the normal bundle of a general elliptic curve \(D \subset \PP^3\) of degree \(4\).

In the case $(g,d) = (8,11)$, it is cited for the following: for three general points $q_3, x, z \in D$, the bundle $N_D[q_3 + x \negmod  z]$ satisfies interpolation. But since $q_3, x \in D$ are general, the lines $\overline{q_3z}$ and $\overline{xz}$ lie in distinct quadrics containing $D$.  It follows that $N_D[q_3 + x \negmod  z] \cong \OO_D(2)(-x) \oplus \OO_D(2)(-q_3)$ satisfies interpolation.

Similarly, in the case \((g, d) = (7, 10)\), it is cited for the following: for three general points $x, z_3, w_3 \in D$, the bundle $N_D[x \negmod z_3][w_3 \binegmod z_3]$ satisfies interpolation. This again follows from the observation that the lines $\overline{xz_3}$ and $\overline{w_3z_3}$ lie in distinct quadrics containing $D$.
\end{proof}

\noindent
We prove the cases $(g,d) \in \{(6,11), (8,12)\}$ below.

\begin{lemma}\label{lem:4.4}
The twisted normal bundle of a general BN-curve of genus $6$ and degree $11$ in $\PP^4$ satisfies interpolation.
\end{lemma}
\begin{proof}
Let $C \subset \PP^4$ be a general elliptic curve of degree 5, and for positive integers $i \leq 2$ and $j \leq 3$ let $x_i, y_i \in C$ and $p_j, q_j \in C$ be ten general points.  By \cite[Lemma~5.1]{lv21} and Lemma \ref{lem: pull off 1-secant}, interpolation for the twisted normal bundle of a general BN-curve of genus 6 and degree 11 reduces to interpolation for the modified twisted normal bundle 
$$N' = N_C(-1)[p_1 \binegmod  q_1][p_2 \binegmod  q_2][p_3 \binegmod  q_3][2x_1 \posmod  y_1][2x_2 \posmod  y_2].$$
Specialize $y_1$ and $y_2$ to a common point $y \in C$. This induces a specialization of \(N'\) to
\[N'' = N_C(-1)[p_1 \binegmod  q_1][p_2 \binegmod  q_2][p_3 \binegmod  q_3][2x_1 + 2x_2 \posmod  y].\]
As in Remark \ref{rem: exact sequence mod}, projection from $y \in C$ induces an exact sequence
\[0 \to \OO(2y + 2x_1 + 2x_2 - p_1 - q_1 -p_2 - q_2 - p_3 - q_3) \to N'' \to  N_{\overline{C}}(-1)[p_1 \binegmod  q_1][p_2 \binegmod  q_2][p_3 \binegmod  q_3]\to 0\]
where $\overline{C} \subset \PP^3$ is the image of $C$ and $N_{\overline{C}}(-1)$ refers to the $N_{\overline{C}} \otimes \OO_{\overline{C}}(-1)$.  After specializing $q_1$ to $p_2$, $q_2$ to $p_3$, and $q_3$ to $p_1$, the last nonzero bundle in the above sequence specializes to \[N_{\overline{C}}(-1)(-p_1 - p_2 - p_3) \cong \OO_{\overline{C}}(1)(-p_1 - p_2 - p_3)^{\oplus 2}.\]
By upper-semicontinuity and generality of $p_1, p_2, p_3 \in C$, we obtain that $h^1(C,N') = 0$.  Thus, by \cite[Lemma~2.8]{lv21}, $N'$ satisfies interpolation.
\end{proof}

\begin{lemma}\label{lem:4.5}
The twisted normal bundle of a general BN-curve of genus $8$ and degree $12$ in $\PP^4$ satisfies interpolation.
\end{lemma}
\begin{proof}
Interpolation for the twisted normal bundle of a general BN-curve of genus $8$ and degree $11$ is proven in \cite[Section~5.1]{lv21}.  We modify  their argument to prove interpolation for the twisted normal bundle of a general BN-curve of genus 8 and degree 12.

Let $D \subset \PP^4$ be a general elliptic curve of degree 6, and let $x, y, z, q_2, q_3 \in D$ be five general points.  As argued in \cite[Section~5.1]{lv21}, by \cite[Lemmas~5.1~and~5.2]{lv21} we may reduce interpolation for the twisted normal bundle of a general BN-curve of genus 8 and degree 12 to interpolation for the modified twisted normal bundle 
$$N' = N_D(-1)(x+y)[x \negmod  y][2y + q_2 + z \negmod  x][2x + q_3 \negmod  z].$$
The exact sequence associated to projection from $x$ is 
\begin{equation}\label{eq:Deg12Genus8}
0 \to \OO(y - q_3)  \to N' \to N_{\overline{D}}(-1)(x - y - q_2 - z)[x \negmod  y][2x + q_3 \negmod  z] \to 0.
\end{equation}
The last nonzero bundle in \eqref{eq:Deg12Genus8} is isomorphic to $N'' := N_{\overline{D}}(-1)(- y - q_2 - z)[x + q_3 \negmod  z]$.  Let $p \in \overline{D}$ be a general point.  If $h^0(\overline{D},N''(-p)) = 0$, then $h^0(D,N'(-p)) = 0$, and so \cite[Lemma~2.8]{lv21} shows that $N'$ satisfies interpolation.  To prove $h^0(\overline{D},N''(-p)) = 0$, further degenerate $\overline{D}$ to the union of an elliptic curve $D' \subset \PP^3$ of degree 4 and a \(1\)-secant line spanned by general points $u \in D'$ and $v \in \PP^3$.  By Lemma \ref{lem: pull off 1-secant}, it suffices to prove $N_{D'}(-1)(-y - q_2 - z)[x + q_3 \negmod  z][2u \posmod  v]$ satisfies interpolation.  We may further limit $x$ and $q_3$ to $u$ to obtain the bundle 
$$N_{D'}(-1)(-y - q_2 - z)[2u \negmod  z][2u \posmod  v] \cong N_{D'}(-1)(-y - q_2 - z).$$
Since $N_{D'}(-1)(-y - q_2 - z) \cong \OO_{D'}(1)(-y - q_2 - z)^{\oplus 2}$ and $\OO_{D'}(1)(-y - q_2 - z)$ is a general line bundle on $D'$, this proves interpolation via upper semicontinuity.
\end{proof}

\begin{proof}[Proof of Theorem \ref{thm-maininterpolation}] In order to prove Theorem \ref{thm-maininterpolation}, it suffices to verify interpolation for the twisted normal bundle holds in the $8$ cases listed in Theorem \ref{thm: basecases}. Six of these cases are verified in Theorem \ref{thm:LV base cases} and the remaining $2$ cases are verified in Lemmas \ref{lem:4.4} and \ref{lem:4.5}. This concludes the proof.
\end{proof}

\subsection{Geometric Interpolation} Having proven Theorem \ref{thm-maininterpolation}, we now prove a geometric consequence.

\begin{proof}[Proof of Corollary \ref{cor:geom interpolation}]
Let $g$ and $d$ be nonnegative integers such that $\rho(g,4,d) \geq 0$ and 
$$(g,d) \notin \{(5,8), (6,9),(7,10)\}.$$
Let $\Gamma \subset \PP^4$ be a set of 
$n = \big\lfloor\frac{5d-g+1}{3}\big\rfloor$
points which is general subject to the condition that exactly $d$ of them lie on a hyperplane $H$. 
We want to prove that there is a smooth nondegenerate BN-curve of genus $g$ and degree $d$ that passes through $\Gamma$.  

If \((g, d) = (2, 6)\), then the desired claim follows from \cite[Section~9.3]{lv22}
(since a collection of \(9\) points on \(C\) which is general subject to the constraint of containing a hyperplane section gives a general line bundle on~\(C\)). 
Otherwise, if $g \geq 1$, Theorem \ref{thm-maininterpolation} implies there exists a smooth BN-curve of genus $g$ and degree $d$ passing through $\Gamma$.  The same is true if $g = 0$ and $d \equiv 1$ mod $3$ by Remarks \ref{rem-g0} and \ref{rem:char2 exceptions}.  To prove Corollary \ref{cor:geom interpolation}, it suffices to establish the claim for $g = 0$ and $d \not\equiv 1$ mod $3$ as well.

To prove Corollary \ref{cor:geom interpolation} for $g = 0$ and $d \equiv 2$ mod $3$, let $p$ be one of the $d$ points in $\Gamma$ that lie on $H$.  Note that $n - 1 = \frac{5(d-1) + 1}{3}$.  Thus, there exists a smooth BN-curve $C \subset \PP^4$ of genus 0 and degree $d-1$ that passes through $\Gamma \setminus \{p\}$.  Let $L_p \subset \PP^4$ be a general line that passes through $p$ and meets $C$.  Let $D = C \cup L_p$ and $q$ be the node of $D$. Consider the subsheaf $N_{D}' \subset N_D$ of the normal bundle that does not smooth the node.  By the exact sequence
$$0 \to N_D|_{L_p}(-p - q) \to N_D'(-\Gamma) \to N_{C}(-(\Gamma\setminus\{p\})) \to 0,$$
we conclude that $h^1(D,N_D'(-\Gamma)) = 0$.  Hence we may smooth $C \cup L_p$ to a BN-curve $C'$ passing through $\Gamma$.  This proves our claim for $d \equiv 2$ mod $3$.

Let $C, L_p$, and $C'$ be as in the preceding paragraph. Before continuing to the case $d \equiv 0$ mod $3$, we prove that deformations of $C'$ that pass through $\Gamma$ sweep out a divisor in $\PP^4$.  Observe that there are no deformations of $C$ which continue to pass through $\Gamma\setminus \{p\}$.  Deformations of the curve $C \cup L_p$ that do not smooth the node and remain incident to $\Gamma$ therefore sweep out a cone $S \subset \PP^4$ whose vertex is $p$ and base is $C$.
We will show that $C \cup L_p$ has no smoothings in $S$ that preserve the incidence with $\Gamma$.

As a divisor on $S$, we claim that $C$ is linearly equivalent to a hyperplane section. Indeed, blowing up the cone point of $S$, the divisor $C$ restricts trivially to the exceptional divisor. It is thus a rational multiple of the hyperplane class, which for degree reasons must be the hyperplane class.

It follows that $N_{C/S} \cong \OO_{\PP^1}(d-1)$.  Since $d \geq 5$, we have $|\Gamma\setminus\{p\}| = \frac{5(d-1) + 1}{3} \geq d + 1$.  The restriction of $N_{C \cup L_p/S}(-\Gamma)$ to $C$ is therefore a line bundle of negative degree.  It follows that each deformation of $C \cup L_p$ in $S$ that passes through $\Gamma$ does not smooth the node, and so $C' \not\subseteq S$.  This proves our claim.

To prove Corollary \ref{cor:geom interpolation} when $d \equiv 0$ mod $3$, let $q_1, q_2 \in \Gamma$ be two points such that $q_1 \in H$ and $q_2 \notin H$.  Since $n - 2 = \frac{5(d-1) - 1}{3}$, there exists a smooth BN-curve $C \subset \PP^4$ of genus 0 and degree $d - 1$ that passes through $\Gamma \setminus \{q_1, q_2\}$.  By the preceding paragraph, deformations of $C$ that pass through $\Gamma \setminus \{q_1, q_2\}$ sweep out a divisor $D \subset \PP^4$.  By generality of $q_1, q_2$, the line $L$ that they span meets $D$ at a point contained in a smooth deformation $C'$ of $C$.  As in the previous case, the nodal union $C' \cup L$ deforms to a smooth BN-curve passing through $\Gamma$.
\end{proof}

\section{Degeneration Arguments}\label{sec: degeneration arguments}

\subsection{Attaching Bisecant Lines}  Let $C$ be a smooth BN-curve.  Let $D$ be the nodal union of $C$ with a bisecant line $L$ through $x,y \in C$, and let $\nu: C \sqcup L \to D$ denote the  normalization of $D$.
By \cite[Lemma~5.4]{lv22}, $D$ is also a BN-curve.  In this subsection, we study the stability of modifications of the normal bundle of $D$.  We recall that $N_D|_C \cong N_C[x \biposmod y]$.

\begin{lemma}\label{lem:extra juice}
Suppose $N$ is a vector bundle on $D$ equipped with an isomorphism to $N_D$ over an open subset $U$ of $D$ containing $L$.
Write $N_C'$ for the subsheaf of $N|_C$ corresponding to sections that do not smooth the nodes \(x\) and \(y\), so that \(N|_C \simeq N_C'[x \biposmod y]\).
For any real number $r$, if the following two conditions both hold:
\begin{center}
\begin{tabular}{l l}
every rank one subbundle of\ldots \hspace{.5cm} & has slope at most\ldots \\
$N_C'[x \biposmod y]$ & $\mu(N_C'[x \biposmod y]) + r + \frac{2}{3}$ \\
$N_C'[x \biposmod y][x \posmod 2y][y \posmod 2x]$ &  $\mu(N_C'[x \biposmod y][x \posmod 2y][y \posmod 2x]) + r + \frac{1}{3}$ \\
\end{tabular}
\end{center}
then any rank one subbundle of $\nu^* N$ has adjusted slope at most $\mu(N) + r$.
Similarly, if 
\begin{center}
\begin{tabular}{l l}
every rank two subbundle of\ldots \hspace{.5cm} & has slope at most\ldots \\
$N_C'[x \biposmod y]$ & $\mu(N_C'[x \biposmod y]) + r + \frac{1}{6}$ \\
$N_C'[x \biposmod y][x \posmod 2y][y \posmod 2x]$ &  $\mu(N_C'[x \biposmod y][x \posmod 2y][y \posmod 2x]) + r + \frac{1}{3}$ \\
\end{tabular}
\end{center}
then any rank two subbundle of $\nu^* N$ has adjusted slope at most $\mu(N) + r$.  Moreover, it suffices to check the first of these conditions only for rank two subbundles of $N_C'[x \biposmod y]$ that are not subbundles of $N_C'$, provided that every rank two subbundle of $N_C'$ has slope at most $\mu(N_C'[x \biposmod y]) + r + \frac{2}{3}$.
\end{lemma}

We use the final statement about subbundles of $N_C'$ only in the proof of Lemma \ref{lem: genus 10 degree 12}.  Therein, it shows that a destabilizing subbundle of $N_C'$ does not create a destabilizing subsheaf of $N$.

\begin{proof}
Let $x'$ and $y'$ be general points on the tangent lines to $C$ at $x$ and $y$, respectively.  Note 
$$N|_L \cong N_L[x\posmod  x'][y\posmod y'] \cong \OO(1) \oplus \OO(2) \oplus \OO(2).$$  
For each subbundle  $F \subset \nu^*N$, let $F_L := F|_L$ and $F_C := F|_C$.  We have
$$\mu^{\text{adj}}(F) = \mu(F_L) + \mu(F_C) - \frac{1}{\text{rk } F}\big[\codim_F(F_L|_x \cap F_C|_x) + \codim_F(F_L|_y \cap F_C|_y)\big].$$

\smallskip
\paragraph{Suppose first $F$ has rank one.}  If $\mu(F_L) < 2$, then
$$\mu^{\text{adj}}(F) \leq \mu(F_L) + \mu(F_C) \leq 1 + \left(\mu(N|_C) + r + \frac{2}{3}\right) = \mu(N|_C) + \mu(N|_L) + r = \mu(N) + r.$$ 
Otherwise, $F_L \cong \OO(2)$.  In particular, $F_L|_x$ lies in $N_{L \rightarrow \overline{x'y'}}|_x$, where $\overline{x'y'}$ denotes the line spanned by $x'$ and $y'$.  Hence, by \cite[Lemma~8.4]{aly}, either $F_C|_x$ lies in $N_{C \rightarrow \overline{yy'}}|_x$ or $\codim_F(F_L|_x \cap F_C|_x) = 1$.  If $F_C|_x$ lies in $N_{C \rightarrow \overline{yy'}}|_x$, then $F_C$ is not saturated in $N|_C[x \posmod  2y]$.  By symmetry and our hypotheses, we see that 
$$\mu(F_C) -\codim_F(F_L|_x \cap F_C|_x) - \codim_F(F_L|_y \cap F_C|_y) \leq \left(\mu(N|_C[x \posmod  2y][y \posmod  2x]) + r + \frac{1}{3}\right) -1 -1.$$
Hence, $\mu^{\text{adj}}(F) \leq \mu(N|_C[x \posmod  2y][y \posmod  2x]) + r + \frac{1}{3} = \mu(N) + r.$
\smallskip

\paragraph{Next, suppose $F$ has rank two.}   If $\mu(F_L)<2$, then 
$$\mu^{\text{adj}}(F) \leq \mu(F_L) + \mu(F_C) \leq \frac{3}{2} + \left(\mu(N|_C) + r + \frac{2}{3}\right) = \mu(N|_C) + \mu(N|_L) + r + \frac{1}{2} = \mu(N) + r + \frac{1}{2}.$$
Thus $\mu^{\text{adj}}(F) \leq \mu(N) + r$ unless $\mu^{\text{adj}}(F) > \mu(F_L) + \mu(F_C) - \frac{1}{2}$ and $\mu(F_L) > 1$ and $\mu(F_C) > \mu(N|_C) + r + \frac{1}{6}$.
The first of these three inequalities implies \(F_C\) and \(F_L\) agree at both nodes. The second implies $F_L$ cannot be a subbundle of $N_L \cong \mathcal{O}(1)^{\oplus 3}$, i.e., $F_L$ must smooth a node. The third implies $F_C$ is a subbundle of $N_C'$, i.e., $F_C$ does not smooth either node. These three inequalities therefore cannot be simultaneously satisfied.
    
Otherwise, $F_L \cong \OO(2) \oplus \OO(2)$.  This implies $F_L|_x=N_{L \rightarrow \overline{x'y'}}|_x$.  Let $F_C^+$ be the saturation of $F_C$ in the modification $N|_C[x \posmod  2y][y \posmod  2x]$. Then the length of $F_C^+ / F_C$ at $x$ is  $\dim_F(F_L|_x \cap F_C|_x)$, and similarly for the length of $F_C^+ / F_C$ at $y$. Thus $\mu(F_C) = \mu(F_C^+) - \frac{1}{2} [\dim_F(F_L|_x \cap F_C|_x) + \dim_F(F_L|_y \cap F_C|_y)]$, so
$$ \mu(F_C) - \frac{1}{2} [\codim_F(F_L|_x \cap F_C|_x) + \codim_F(F_L|_y \cap F_C|_y)] = \mu(F_C^+) - 2.$$
Therefore,
\[\mu^{\text{adj}}(F) = \mu(F_C^+) \leq \mu(N|_C[x \posmod 2y][y \posmod 2x]) + r + \frac{1}{3} = \mu(N) + r. \qedhere\]
\end{proof}

\begin{remark}\label{rem:iterative extra juice}
Lemma \ref{lem:extra juice} may be applied iteratively.  For instance, suppose $D = C \cup L_1 \cup L_2$, where $L_i$ is a two-secant line through general points $x_i, y_i \in C$. Write \(N = N_C[x_1 \biposmod y_1][x_2 \biposmod y_2] = N_D|_C\). By applying  Lemma \ref{lem:extra juice} first to the union of a smoothing of \(C \cup L_1\) with \(L_2\), then with \(C \cup L_1\), we obtain the following: for any real number $r$, if\ldots
\begin{center}
\begin{tabular}{l l}
every rank one subbundle of\ldots \hspace{.5cm} & has slope at most\ldots \\
$N$ & $\mu(N) + r + \frac{4}{3}$ \\
$N[x_1 \posmod  2y_1][y_1 \posmod  2x_1]$ & $\mu(N[x_1 \posmod  2y_1][y_1 \posmod  2x_1]) + r + 1$ \\
$N[x_2 \posmod  2y_2][y_2 \posmod  2x_2]$ & $\mu(N[x_2 \posmod  2y_2][y_2 \posmod  2x_2]) + r + 1$ \\
$N[x_1 \posmod  2y_1][y_1 \posmod  2x_1][x_2 \posmod  2y_2][y_2 \posmod  2x_2]$ & $\mu(N[x_1 \posmod  2y_1][y_1 \posmod  2x_1][x_2 \posmod  2y_2][y_2 \posmod  2x_2]) + r + \frac{2}{3}$ \\
\end{tabular}
\end{center}
\ldots then any rank one subbundle of a general deformation of \(D\) has slope at most $\mu(N_D) + r$. Of course by symmetry the two middle conditions are equivalent.
In addition, similar formulas hold for rank two subbundles.
\end{remark}

\subsection{Modifications of Bundles on Elliptic Curves}
To apply Lemma \ref{lem:extra juice}, we show related modifications of normal bundles of elliptic curves are \textit{almost} stable.  To begin, we have the following result from \cite{cs23}.

\begin{theorem}[Theorem 3.1 and Proposition 3.3 \cite{cs23}]\label{thm: normal bundle elliptic}
Let $C \subset \PP^r$ be a general $BN$-curve of genus $g = 1$ and degree $d \geq r + 1 \geq 4$.  Then $N_C$ is semistable.  Moreover, if $\mu(N_C) \in \mathbb{Z}$ and $r = 4$, then $N_C$ is a direct sum of distinct line bundles.
\end{theorem}

If $r = 3$, then $N_C$ need not be a direct sum of distinct line bundles.  Indeed, \cite[Theorem~2]{l24} implies $N_C$ is indecomposable if $d$ is odd and the characteristic of the base field is 2.  To ensure our arguments remain independent of characteristic, we will use only the following lemma to describe these normal bundles.

\begin{lemma}\label{lem: elliptic in P3}
Let $C \subset \PP^3$ be a general $BN$-curve of genus $g = 1$ and degree $d \geq 5$.  Regardless of characteristic, $N_C$ is not a direct sum of two isomorphic line bundles.
\end{lemma}
\begin{proof}
We must show $\dim_k \Hom(\mathcal{L}, N_C) \leq 1$ for any line bundle $\mathcal{L}$ of slope $\mu(N_C)$.  We prove this claim by induction on $d$.  When $d = 5$, this claim follows from \cite[Theorem~2]{l24} in characteristic 2 and \cite[Lemma~10.1]{l21} in all other characteristics. 
    
Suppose $d > 5$ instead and let $C \subset \mathbb{P}^3$ be a general BN-curve of genus $1$ and degree $d - 1$.  Let $C \cup L$ be the union of $C$ with a general one-secant line through $u \in C$.  As $h^1(N_L(-u)) = h^1(N_C) = 0$, the curve $C\cup L$ is strongly smoothable.  Consider a one-parameter family with smooth total space smoothing $C \cup L$ to  a general BN-curve $C_1$ of genus \(1\) and degree $d$.  Let \(v \neq u\) be a point on \(L\).  As shown in Lemma \ref{lem: pull off 1-secant}, $N_{C_1/\mathbb{P}^3}$ is a generalization of $N_C[2u \posmod  v](u)$.  By the inductive hypothesis and generality of $L$, the bundle $N_C[u \posmod  v]$ is stable.  This implies  $\dim_k \Hom(\mathcal{L}, N_C[2u \posmod  v](u)) \leq 1$ for each line bundle $\mathcal{L}$ of slope $\mu(N_C) + 2$, which proves our claim through upper semicontinuity.
\end{proof}

\begin{definition}\label{defn: even modification}
Let $C \subset \PP^r$ be a smooth curve.  Consider a modification $E$ of $N_C$ such that the modification data depends only on a general collection of points $\{p_i\} _{i\in I} \subset C$ and a general point in a rational variety $V$.
Suppose there is an action of an algebraic group $G$ on $I \times V$ that preserves the modification data.  If the cardinality of every orbit of the induced $G$ action on $I$ is divisible by $m$, then we say that $E$ is an \defi{$m$-cyclic modification} of $N_C$.
\end{definition}

\begin{lemma}\label{lem:black magic}
Let $C\subset \PP^r$ be a general $BN$-curve of genus $g = 1$ and degree $d$ divisible by $m$.  Suppose $E$ is an $m$-cyclic modification of $N_C$. 
Let $E_s/E_{s-1}$ be one of the factors in the HN-filtration of $E$.  Consider the associated-graded module $\bigoplus_{t \in T} V_{s,t}^{\oplus n_{s,t}}$ to a JH-filtration of $E_s/E_{s-1}$, written so that the $V_{s,t}$ are nonisomorphic.  
For any $(s, n)$, write \(h_{s, n}  = \deg (\bigoplus_{\{t \in T| n_{s,t} = n\}} V_{s,t})\). Then \(m \mid h_{s, n}\). 
\end{lemma}

\begin{proof}
For a fixed curve \(C\) of genus \(1\), let \(X\) be the parameter space of such bundles \(E\). In other words, \(X\) parameterizes embeddings of \(C\) in \(\mathbb{P}^r\), together with choices of points \(p_i \in C\) and a point in the rational variety $V$. There is a nonempty open subset of $X$ where the numerical invariants of the HN and JH filtrations of $E$ are constant. The first Chern class of \(\bigoplus_{\{t \in T| n_{s,t} = n\}} V_{s,t}\) thus defines a rational map $f_{s, n}\colon  X \dashrightarrow \text{Pic}^{h_{s,n}}(C)$.

Since $\text{Pic}^{h_{s,n}}(C)$ is a torsor for an abelian variety, $f_{s, n}$ extends to a regular map and factors through a map $\bar{f}_{s, n} \colon \text{Pic}^d(C) \times \text{Pic}^1(C) \times \cdots \times \text{Pic}^1(C) \to \text{Pic}^{h_{s,n}}(C)$. Because maps between torsors for abelian varieties are group homomorphisms up to translation, we can write \(\bar{f}_{s, n} = \varphi + \psi_1 + \psi_2 + \cdots + \psi_{|I|}\), where \(\varphi : \text{Pic}^d(C) \to \text{Pic}^{h_{s,n}}(C)\) and the \(\psi_i : \text{Pic}^1(C) \to \text{Pic}^{h_{s,n}}(C)\) are group homomorphisms up to translation.
By \(G\)-invariance, \(\psi_i\) and \(\psi_j\) differ by translation for $i,j \in I$ in the same $G$-orbit.
Consider the map \(t : C \to C\) given by translation by a fixed \(m\)-torsion point.
Since \(d\) and the cardinalities of all the orbits of \(G\) acting on \(I\) are divisible by \(m\), we have \(\bar{f}_{s, n} = \bar{f}_{s, n} \circ t^*\). Since \(\bar{f}_{s, n} \circ t^* = t^* \circ \bar{f}_{s, n}\), it follows that \(t^*\) must act trivially on \(\text{Pic}^{h_{s,n}}(C)\). Thus, $m \mid h_{s, n}$.
\end{proof}

We use Lemma \ref{lem:black magic} below to study modified normal bundles arising from the attachment of one-secant and two-secant lines.  This is explained in the following section. 

\section{Types of Bundles}\label{sect: types of bundles}
Throughout this section, we let $C \subset \PP^4$ be a general elliptic curve of even degree.  To prove Theorem \ref{thm-mainstab}, we will consider normal bundles of nodal BN-curves $D$ obtained by gluing one- and two-secant lines to $C$.  Lemmas \ref{lem: pull off 1-secant} and \ref{lem:extra juice} allow us to prove $N_D$ is stable by studying several modifications of its restriction to $C$.  By Lemma \ref{lem:extra juice} it suffices to prove many of these modifications are ``almost" stable.  In this section we bound ``how unstable'' these modifications can be. 

To do this, suppose some modification $E$ of $N_C$ is unstable, and consider the effect of adding the modifications appearing in Lemmas \ref{lem: pull off 1-secant} or \ref{lem:extra juice} to $E$.
Because the directions of these modifications are linearly general, they are transverse to each subbundle of $E$ appearing in the HN-filtration.  In particular, if $E$ is ``sufficiently unstable,'' then this new modification is ``less unstable.'' In this section we make this qualitative argument precise by defining several different ``types'' of possible HN-filtrations and showing that adding modifications appearing in Lemmas \ref{lem: pull off 1-secant} and \ref{lem:extra juice} produces a bundle with one of finitely many types.

\begin{definition}\label{def-types_of_bundles}
Let $E$ be a modification of the normal bundle of a general BN-curve $C \subset \PP^4$ of genus $1$ and even degree.
We say that $E$ has type $\Stype{c}$ if $c_1(E) \equiv c$ mod $3$, $E$ is semistable, and $E$ is not isomorphic to $L^{\oplus 3}$ for any line bundle $L$ (this last condition is immediate unless $c = 0$).
    
Furthermore, we say that $E$ has type $\Utype{c}{r}{\delta}$ if $c_1(E) \equiv c$ mod $3$, and $E$ has a two-step HN-filtration $0 \subset V \subset E$ where $\rk(V) = r$ and \(\mu(V) = \mu(E) + \delta\).
\end{definition}
\noindent
Our $7$ types of HN-filtrations will be: 
$\Stype{0}$, $\Stype{1}$, $\Stype{2}$, $\Utype{1}{1}{\frac{2}{3}}$, $\Utype{1}{2}{\frac{1}{6}}$, $\Utype{2}{1}{\frac{1}{3}}$, $\Utype{2}{2}{\frac{1}{3}}$.

The following two lemmas describe how modifications $[x \posmod y]$ and $[x \posmod 2y]$ meet subbundles of $E$, and will form the backbone of our approach in this section.

\begin{lemma}\label{lem: direction of modification}
Fix a point $x \in C$ lying in the open subset on which $E$ is isomorphic to $N_C$.  As $y \in C$ varies, the directions of the modifications $[x \posmod y]$ and $[x \posmod 2y]$ in $\PP E|_x$ vary in linearly general families.
\end{lemma}
\begin{proof}
Since $C \subset \PP^4$ is linearly nondegenerate, the direction of the modification $[x \posmod y]$ in $\PP E|_x$ is linearly general.
In fact, the map $C \dashrightarrow \PP E|_x$ given by $y \mapsto [x \posmod y] \in \PP E|_x$ is naturally the projection of $C$ from its tangent line $T_xC$ at $x$.

Furthermore, we claim that the image of $C \dashrightarrow \PP E|_x$ is not a strange curve. Indeed, applying Lemma~\ref{lem: reducible BN-curves} to specialize \(C\) to the union of an elliptic normal curve and several \(1\)-secant lines induces a specialization of this projection to a reducible curve, one component of which is a smooth plane cubic. Thus the modification $[x \posmod 2y]$ is also linearly general in the dual space $\PP E|_x^\vee$.
\end{proof}

\begin{lemma}\label{lem: JHF images} 
Suppose $E$ is semistable.  Let $p \in \PP E|_x$ be a general point in some irreducible and linearly general subvariety of a fiber. If there exists a proper subbundle $F \subset E$ of slope $\mu(F) = \mu(E)$ such that $p \in \PP F|_x \subset \PP E|_x$, then:
\begin{enumerate}
\item If $\rk(F) = 1$, then $E \cong F^{\oplus 3}$;
\item If $\rk(F) = 2$, then there exist line bundles $L_1, L_2$ such that $E$ is an extension of $L_2 \oplus L_2$ by $L_1$:
$$0 \rightarrow L_1 \rightarrow E \rightarrow L_2 \oplus L_2 \rightarrow 0.$$
In this case, if $E \not\cong L_2^{\oplus 3}$, then the extension above is canonical and $F$ must contain $L_1$; if in addition $L_1 \cong L_2$, then $F$ is the nonsplit extension of $L_1 \cong L_2$ by itself.
\end{enumerate}
\end{lemma}
\begin{proof}
Subbundles \(F \subset E\) of slope \(\mu(F) = \mu(E)\) have one of finitely many isomorphism classes.  Thus we can suppose the isomorphism class of \(F\) does not depend on \(p\).

If \(\rk(F) = 1\), then \(\dim_k \Hom(F, E) = 3\), so \(E \simeq F^{\oplus 3}\) as claimed.

Suppose $\rk(F) = 2$ instead.  Let $L_2 = \bigwedge^3 E \otimes \bigwedge^2 F^\vee$.  Observe that any inclusion $F \subset E$ corresponds to a surjection $E \to L_2$.  Since $p$ varies in an irreducible linearly general family, $\dim_k \Hom(E, L_2)$ is $2$ or $3$. If $\dim_k \Hom(E, L_2) = 3$ then $E \simeq L_2^{\oplus 3}$. Otherwise there is a surjection $\theta : E \to L_2 \oplus L_2$ that is canonical up to automorphisms of the target.  In this case $\ker (\theta)$ is a line bundle \(L_1\) that must be contained in $F$.  

If $L_1 \cong L_2$ and $E \not\cong L_2^{\oplus 3}$, then $\Hom(L_2, E) = 2$, and so there is a unique subbundle isomorphic to $L_2^{\oplus 2}$. Thus $F \not\cong L_2^{\oplus 2}$ by linear generality of $p$.  We conclude $F$ is the nonsplit extension of $L_1 \cong L_2$ by itself.
\end{proof}

\begin{corollary}\label{cor: JHF images dual}
Suppose $E$ is semistable.  Let $\Lambda \subset \PP E|_x$ be a general line in some irreducible and linearly general subvariety of $\PP E|_x^\vee$.  If there exists a proper subbundle $F \subset E$ of slope $\mu(F) = \mu(E)$ such that $\Lambda$ contains $\PP F|_x$, then:
\begin{enumerate}
\item If $\rk(F) = 2$, then $E \cong (E/F)^{\oplus 3}$;
\item If $\rk(F) = 1$, then there exist line bundles $L_1, L_2$ such that $E$ is an extension of $L_2$ by $L_1 \oplus L_1$:
$$0 \rightarrow L_1 \oplus L_1 \rightarrow E \rightarrow L_2 \rightarrow 0.$$
In this case, if $E \not\cong L_1^{\oplus 3}$, then the extension above is canonical and $F$ is contained in $L_1 \oplus L_1$; if in addition $L_1 \simeq L_2$, then $E/F$ is the nonsplit extension of $L_1 \simeq L_2$ by itself.
\end{enumerate}
\end{corollary}
\begin{proof}
Apply Lemma~\ref{lem: JHF images} to $E^\vee$.
\end{proof}

\subsection{Two-Cyclic Modifications}
Adding the modifications appearing in Lemma~\ref{lem:extra juice} preserves two-cyclicity.
Therefore, by considering these modifications first, we can assume while studying them that $E$ is two cyclic.  Since $C$ is assumed to have even degree, this allows us to apply Lemma \ref{lem:black magic}.  In this section we use the following additional assumptions and notation:
\begin{enumerate}
\item $x,y \in C$ are a general pair of points;
\item $E$ is a two-cyclic modification of the normal bundle of $C$;
\item $\ell$ is the unique integer such that $\mu(E) \in \{2\ell, 2\ell + \frac{2}{3}, 2\ell + \frac{4}{3}\}$.
\end{enumerate}
Note that if $E$ is of type $S_c$ or $U_c(r, \delta)$, then
\[\mu(E) = \begin{cases}
2\ell & \text{if $c = 0$,} \\
2\ell + \frac{4}{3} &  \text{if $c = 1$,} \\
2\ell + \frac{2}{3} & \text{if $c = 2$.}
\end{cases}\]
By Lemma \ref{lem:black magic} only the following $5$ types may be two-cyclic:
$\Stype{0}$, $\Stype{1}$, $\Stype{2}$, $\Utype{1}{1}{\frac{2}{3}}$, $\Utype{2}{2}{\frac{1}{3}}$. 

\begin{remark}\label{rem: destab quotient distinct JHF factors}
For $E$ a two-cyclic modification of $N_C$, if $E$ has type $\Utype{1}{1}{\frac{2}{3}}$ and $V \subset E$ is the maximal destabilizing subbundle, then Lemma~\ref{lem:black magic} implies $E/V$ is a direct sum of two non-isomorphic line bundles. 
\end{remark}

Observe that the first Chern classes of $E$ and $E[x \biposmod  y][x \posmod  2y][y \posmod 2x]$ are congruent modulo 3.  By choosing an appropriate order of modifications, we will only apply the modification $[x \biposmod  y][x \posmod  2y][y \posmod 2x]$ to bundles $E$ of type $S_0$.  The following lemma shows this does not change the bundle's type.  

\begin{lemma}\label{lem: S_0 big mod}
Suppose $E$ has type $\Stype{0}$.  Then $E[x \biposmod  y][x \posmod  2y][y \posmod 2x]$ has type $\Stype{0}$ as well.
\end{lemma}

\begin{proof}
Suppose $E[x \biposmod  y][x \posmod 2y][y \posmod 2x]$ is not semistable, and let $F$ be the maximal destabilizing subbundle. Since $E$ is semistable: If $\rk(F)=1$, then $\mu(F) \leq 2\ell+4$; and if $\rk(F)=2$, then $\mu(F) \leq 2\ell+3$. By Lemma \ref{lem:black magic}: If $\rk(F)=1$, then $\mu(F) \in 2 \mathbb{Z}$; and if $\rk(F)=2$, then $\mu(F) \in \mathbb{Z}$. It thus suffices to rule out the case of equality, i.e., when there is a line subbundle (respectively a rank 2 subbundle) of $E$, with the same slope as $E$, and which picks up as many of the modifications at $x$ and $y$ as possible.  These cases are ruled out by Lemmas \ref{lem: direction of modification} and \ref{lem: JHF images}, since $E$ is not isomorphic to $L^{\oplus 3}$ for any line bundle $L$.

Next, we rule out that $E' := E[x \biposmod  y][x \posmod  2y][y \posmod 2x] \cong L^{\oplus 3}$ for some line bundle $L$. Suppose for the sake of contradiction that this were true. Since $E'|_x \to E(2x + 2y)|_x$ and $E'|_y \to E(2x + 2y)|_y$ are of rank $1$, their kernels would give distinguished lines in \(\PP E'|_x \simeq \PP E'|_y\).  The negative modification of $E'$ towards these lines is $I := E[x \posmod  2y][y \posmod 2x]$. 

If these two lines coincide, then $I \simeq L^{\oplus 2} \oplus L(-x-y)$. 
This would force $E$ to have a rank two subbundle of slope $\mu(E)$ that aligns with both modifications $[x \posmod 2y], [y \posmod 2x]$.  By Corollary \ref{cor: JHF images dual} this cannot happen.  

If these two lines differ, then $I \simeq L \oplus L(-x) \oplus L(-y)$.  The only way this can happen is if there is a subbundle $F \subset E$ such that $F \cong L(-x -y)$ and $F$ meets both of the positive modifications $[x \posmod 2y], [y \posmod 2x]$.  Letting $L_1$ denote $L(-x - y)$, by Corollary \ref{cor: JHF images dual} this forces $E$ to be an extension of some line bundle $L_2$ by $L_1 \oplus L_1$.  A Chern class computation shows $L_2 \cong L_1$.  By Corollary \ref{cor: JHF images dual} this implies that $E/F$ is the nonsplit extension of $L_1$ by itself.  
    
In particular, there is a unique map up to scaling from $L(-x-y)$ to $E/F$.  Composing with the natural inclusion $E/F \hookrightarrow I/L \cong L(-x) \oplus L(-y)$ we obtain a nonzero map $L(-x -y) \to L(-x) \oplus L(-y)$.  Equivalently, we obtain maps $L(-x -y) \to L(-x)$ and $L(-x -y) \to L(-y)$ that are not both zero.  By symmetry/monodromy swapping $x$ and $y$, both of these maps must be nonzero.  We conclude that the image of $L(-x-y)|_x$ is the first coordinate subspace, while the image of $L(-x-y)|_y$ is the second coordinate subspace.  Recalling that $E/F$ is obtained from negative modifications of $I/L$ at $x$ and $y$, these negative modifications must be towards the first coordinate subspace at $x$ and the second coordinate subspace at $y$ in order for the map $L(-x-y) \to I/L$ to factor through $E/F$ as constructed. But this implies $E/F \simeq L(-x-y)^{\oplus 2}$, contradicting that $E/F$ is the nonsplit extension of $L(-x-y)$ by itself.
\end{proof}

\noindent
Next, we study the type of $E[x \biposmod  y]$ given the type of $E$ in Lemmas \ref{lem: S_0 small mod} through \ref{lem: S2 U22 small mod}.

\begin{lemma}\label{lem: S_0 small mod}
Suppose $E$ has type $\Stype{0}$.  Then $E[x \biposmod  y]$ has type $\Stype{2}$ or $\Utype{2}{2}{\frac{1}{3}}$.
\end{lemma}
\begin{proof}
Suppose $E[x \biposmod  y]$ is not stable. Let $F$ be the maximal destabilizing subbundle. Since $E$ is semistable, if $\rk(F)=2$, then $\mu(F) \leq 2\ell+1$ as desired. Similarly, if $\rk(F)=1$, then $\mu(F) \leq 2\ell+2$, and $\mu(F) \neq 2\ell + 1$ by Lemma \ref{lem:black magic}. If $\mu(F)=2\ell+2$, then $E$ must have a line subbundle that picks up both modifications. Since $E$ is not equal to $L^{\oplus 3}$ and the modification $[x \posmod y]$ is general, this is impossible by Lemma \ref{lem: JHF images}.
We conclude that either $E[x \biposmod  y]$ is stable or the maximal destabilizing $F$ has $\rk(F)=2$ and $\mu(F)= \mu(E) + \frac{1}{3}$.
\end{proof}

\begin{remark}\label{rem: even mod distinct JHF}
The argument of Lemma \ref{lem: S_0 small mod} shows that $E[x \biposmod  y]$ is stable if $\dim_k \Hom(L, E) \leq 1$ for every line bundle $L$ of slope $\mu(E)$. 
Indeed, in this case, there can only be finitely many rank 2 subbundles of $E$ with this slope and none can pick up both modifications.
\end{remark}

\begin{lemma}\label{lem: S1 U11 U12 small mod}
Suppose $E$ has type $\Stype{1}$ or $\Utype{1}{1}{\frac{2}{3}}$.  Then $E[x \biposmod  y]$ has type $\Stype{0}$.
\end{lemma}
\begin{proof}
Suppose $E [x \biposmod  y]$ is not semistable. Let $F$ be the maximal destabilizing subbundle. Any line subbundle of $E$ has slope at most $2\ell + 2$ and any rank two subbundle of $E$ has slope at most $2\ell + \frac{3}{2}$.  So, if \(\rk(F) = 1\), then \(\mu(F) \leq 2\ell + 4\), and if \(\rk(F) = 2\), then \(\mu(F) \leq 2\ell + \frac{5}{2}\). By Lemma \ref{lem:black magic}, if \(\rk(F) = 1\), then \(\mu(F) \in 2\mathbb{Z}\), and if \(\rk(F) = 2\), then \(\mu(F) \in \mathbb{Z}\).  We conclude that $\rk(F) = 1$, and $E$ has type $\Utype{1}{1}{\frac{2}{3}}$, and the intersection of $F$ with $E \subset E [x \biposmod  y]$ is the maximal destabilizing subbundle of $E$, which must pick up both modifications.  Thus, it suffices to observe that the maximal destabilizing subbundle of \(E\) cannot pick up both modifications by generality of the direction \([x \posmod y]\).

Next we rule out $E[x \biposmod y] \cong L^{\oplus 3}$ for some line bundle $L$. If this were the case, then as in the proof of Lemma~\ref{lem: S_0 big mod}, either $E \cong L \oplus L(-x) \oplus L(-y)$ or $E \cong L\oplus L \oplus L(-x-y)$. In either case, the line bundle $\OO_C(x+y)$ would be determined by the isomorphism class of $E$, contradicting the generality of $x$ and $y$.
\end{proof}

\begin{lemma}\label{lem: S2 U22 small mod}
Suppose $E$ has type $\Stype{2}$ or $\Utype{2}{2}{\frac{1}{3}}$.  Then $E[x \biposmod  y]$ has type $\Stype{1}$ or $\Utype{1}{1}{\frac{2}{3}}$.
\end{lemma}
\begin{proof}
Suppose $E [x \biposmod  y]$ is not stable.  Let $F$ be the maximal destabilizing subbundle. Any subbundle of $E$ has slope at most $2\ell+1$. So, if \(\rk(F) = 1\), then \(\mu(F) \leq 2\ell + 3\), and if \(\rk(F) = 2\), then \(\mu(F) \leq 2\ell + 2\). Furthermore, in case of equality, \(E\) must have type \(\Utype{2}{2}{\frac{1}{3}}\) and \(F\) must arise from a bundle contained in the maximal destabilizing subbundle which picks up as many modifications as possible. This cannot happen by Lemma \ref{lem: direction of modification}. As $\rk(F)\cdot \mu(F)$ cannot be odd by Lemma \ref{lem:black magic}, we conclude that $\rk(F) = 1$ and $\mu(F) = 2\ell + 2$.
Properties of the HN-filtration imply the quotient by $F$ is semistable, and so $E[x \biposmod  y]$ has type $\Utype{1}{1}{\frac{2}{3}}$. 
\end{proof}

\noindent
The following proposition summarizes our results.

\begin{proposition}\label{prop: types of bundles summary}
Let $C \subset \PP^4$ be a general BN-curve of even degree and genus $g = 1$. Suppose $E$ has type $\Stype{0}$.
Consider the modification 
$$E_{n,m} = E[x_1 \biposmod  y_1][x_1 \posmod  2y_1][y_1 \posmod 2x_1]\ldots [x_n \biposmod  y_n][x_n \posmod  2y_n][y_n \posmod 2x_n][z_1 \biposmod  w_1]\ldots [z_m \biposmod  w_m]$$
where $n,m \geq 0$ and the $x_i,y_i,z_j$ and $w_j$ form a general collection of points on $C$.  The bundle $E_{n,m}$ has one of the \(5\) types $\Stype{0}$, $\Stype{1}$, $\Stype{2}$, $\Utype{1}{1}{\frac{2}{3}}$, $\Utype{2}{2}{\frac{1}{3}}$  described by Definition \ref{def-types_of_bundles}. 
\end{proposition}
\begin{proof}
Applying Lemma \ref{lem: S_0 big mod}, we conclude that \(E_{n,0}\) has type \(S_0\).

We then apply Lemmas \ref{lem: S_0 small mod} through \ref{lem: S2 U22 small mod} to study the modifications \([z_j \biposmod w_j]\). The following table displays the possible transitions between \(E_{n,m}\) and \(E_{n,m+1}\), where the rows correspond to the possible types of \(E_{n,m}\) and the columns correspond to the possible types of \(E_{n,m+1}\).

\begin{center}
{\renewcommand{\arraystretch}{1.2}
\begin{tabular}{c||c|c|c|c|c}
{} & $\Stype{0}$ & $\Stype{1}$ & $\Stype{2}$ & $\Utype{1}{1}{\frac{2}{3}}$ & $\Utype{2}{2}{\frac{1}{3}}$ \\ \hline\hline
$\Stype{0}$ & & & $\bullet$ & & $\bullet$ \\ \hline 
$\Stype{1}$ & $\bullet$ & & & & \\ \hline
$\Stype{2}$ & & $\bullet$ & & $\bullet$ & \\ \hline
$\Utype{1}{1}{\frac{2}{3}}$ & $\bullet$ & & & & \\ \hline
$\Utype{2}{2}{\frac{1}{3}}$ & & $\bullet$ & & $\bullet$ &  
\end{tabular}}
\end{center}
This completes the proof.
\end{proof}

\subsection{Attaching One-Secant Lines} Here we describe some stability properties of bundles obtained by gluing one-secant lines to the elliptic curve $C$.  Throughout this subsection we adopt the following notation:
\begin{itemize}
\item $(u,v) \in C \times \PP^4$ is a general point;
\item $E$ is a vector bundle on $C$ such that either $E$ or $E[2u \posmod  v]$ is a two-cyclic modification of $N_C$;
\item $\ell$ is the unique integer such that $\mu(E) \in \{2\ell, 2\ell + \frac{2}{3}, 2\ell + \frac{4}{3}\}$.
\end{itemize}

In this subsection, we will study the modification $E[2u \posmod  v] \cong E[u \posmod v][u \posmod v]$ that arises in Lemma~\ref{lem: pull off 1-secant} from attaching the one-secant line $\overline{uv}$ to $C$.
We will assume $E$ has one of the $7$ types listed after Definition~\ref{def-types_of_bundles}.  By Lemma \ref{lem:black magic}, $E$ cannot have type $\Utype{1}{2}{\frac{1}{6}}$ or $\Utype{2}{1}{\frac{1}{3}}$ if it is a two-cyclic modification of $N_C$.  

\begin{lemma}\label{lem: S_0 one-secant mod}
Suppose $E$ has type $\Stype{0}$.  Then $E[2u \posmod v]$ has type $\Stype{2}$, $\Utype{2}{2}{\frac{1}{3}}$, or $\Utype{2}{1}{\frac{1}{3}}$.
\end{lemma}
\begin{proof}
Suppose $E[2u \posmod v]$ is not stable.  Let $F$ be the maximal destabilizing subbundle. Since $E$ is semistable, if $\rk(F)=2$, then $\mu(F) \leq 2\ell+1$ and $E[2u \posmod v]$ has type $\Utype{2}{2}{\frac{1}{3}}$. Similarly, if $\rk(F)=1$, then $\mu(F) \leq 2\ell+2$. If $\mu(F)=2\ell+2$, then $E$ must have a line subbundle that picks up both modifications. Since $E \not\cong L^{\oplus 3}$ for any line bundle $L$ and the modification $[u \posmod v]$ is general, this is impossible by Lemma \ref{lem: JHF images}.   
We conclude that either $E[2u \posmod v]$ is stable or the maximal destabilizing subbundle $F$ has slope $\mu(F)= \mu(E) + \frac{1}{3}$.  If $\rk(F) = 1$, then properties of the HN-filtration imply the quotient by $F$ is semistable, and so $E[2u \posmod v]$ has type $\Utype{2}{1}{\frac{1}{3}}$. 
\end{proof}

\begin{remark}
The argument of Lemma \ref{lem: S_0 one-secant mod} shows that $E[2u \posmod  v]$ has type $\Stype{2}$ or $\Utype{2}{1}{\frac{1}{3}}$ whenever $\dim_k \Hom(L, E) \leq 1$ for every line bundle $L$ of slope $\mu(E)$. 
Indeed, in this case, there can only be finitely many rank $2$ subbundles of $E$ with slope $\mu(E)$ and none can pick up the modification $[u \posmod v]$.
\end{remark}

\begin{lemma}\label{lem: S1 one secant mod}
Suppose $E$ has type $\Stype{1}$, $\Utype{1}{1}{\frac{2}{3}}$, or $\Utype{1}{2}{\frac{1}{6}}$.  Then $E[2u \posmod  v]$ has type $\Stype{0}$.
\end{lemma}
\begin{proof}
Suppose $E[2u \posmod v]$ is not semistable.  Let $F$ be the maximal destabilizing subbundle.  Any subbundle of $E$ has slope at most $2\ell+2$. Furthermore, the maximal subbundle of slope $2\ell+2$ is unique. Such a bundle therefore misses the modification $[2u \posmod v]$ entirely.   We conclude that if $\rk(F)=1$, then $\mu(F) \leq 2\ell+3$. If $\rk(F) = 2$, then $\mu(F) \leq 2\ell + \frac{5}{2}$.  Since $\mu(F) > 2\ell + 2$, in both cases we must have equality.  Therefore, $c_1(F)$ is odd and so Lemma \ref{lem:black magic} implies $E[2u \posmod v]$ cannot be a two-cyclic modification of $N_C$. By hypothesis $E$ must be a two-cyclic modification of $N_C$.  In particular, $E$ must have type $\Stype{1}$ or $\Utype{1}{1}{\frac{2}{3}}$.      
Consider the intersection $F'$ of $F$ with $E \subset E [2u \posmod v]$.  Since $F'$ cannot have slope $2\ell + 2$ and $\mu(F) > 2\ell + 2$, we must have $c_1(F') = c_1(F) - 2$.
    
First suppose that $\rk(F) = 2$, so that $F' \subset E$ is a subbundle of slope $2\ell + \frac{3}{2}$.  As $\mu(E) < \mu(F')$, the bundle $E$ is not stable and $F'$ must intersect the maximal destabilizing subbundle $V \subset E$.  Thus $E$ has type $\Utype{1}{1}{\frac{2}{3}}$.  Then by Remark~\ref{rem: destab quotient distinct JHF factors}, $E \cong V \oplus L_1 \oplus L_2$ for some pair of line bundles $L_1 \not\cong L_2$, and $F'$ would be a direct summand of $E$. This contradicts the generality of $v \in \PP^4$.
    
Next suppose that $\rk(F) = 1$ and $F' \subset E$ is a line bundle of degree $2\ell + 1$.  The isomorphism class of $F'$ cannot depend on the choice of $v \in \PP^4$ because the latter is parameterized by a rational variety.  Thus, since $F'$ meets the linearly general modification $[2u \posmod v]$ we must have $\dim_k \Hom(F', E) \geq 3$. This cannot happen by Remark \ref{rem: destab quotient distinct JHF factors}.  Thus $E [2u \posmod v]$ is semistable.
    
It remains to show that $E [2u \posmod v]$ is not isomorphic to $L^{\oplus 3}$ for any line bundle $L$.  If this were the case,  
then there would be a nonzero map $L \to E$ that destabilizes $E$.  In this case $L$ would have to be a JH-factor of the maximal destabilizing subbundle of $E$.  In particular, the first Chern class of $L$ is independent of $u$.  The line bundle $\OO_C(2u)$ would therefore be determined by the isomorphism class of $E$, contradicting generality of $u$.
\end{proof}

\begin{lemma}\label{lem: S2 one secant mod}
If $E$ has type $\Stype{2}$, then $E[2u \posmod v]$ has type $\Stype{1}$.
\end{lemma}

\begin{proof}
Suppose $E[2u \posmod v]$ is unstable and let $F$ be its maximal destabilizing subbundle.  Denote by $F'$ the intersection of $F$ with $E \subset E[2u \posmod v]$.  Since $E$ is stable, either $\rk(F) = 2$ and $\mu(F) = 2\ell + \frac{3}{2}$, or $\rk(F) = 1$ and $\mu(F) = 2\ell + 2$.  In both cases, $c_1(F') = c_1(F) - 2$ and $F'$ is stable.  Since $v$ is parameterized by a rational variety, the classification of stable vector bundles on elliptic curves implies the isomorphism type of $F'$ is independent of $v$.  Since $E$ is stable as well, $\dim_k \Hom(F',E) = \chi(\Hom(F',E)) = \frac{2}{\rk(F')}$.  Hence, for $v \in \PP^4$ general, $N_{C \to v}|_u$ lies outside the linear span of $\Hom(F',E)|_u \subset E|_u$, contradicting $c_1(F') = c_1(F) - 2$.
\end{proof}

\begin{lemma}
If $E$ has type $\Utype{2}{1}{\frac{1}{3}}$, then $E[2u \posmod v]$ has type $\Stype{1}$.
\end{lemma}

\begin{proof}
Any line subbundle of $E$ has degree at most $2\ell + 1$, and any rank two subbundle of $E$ has slope at most $2\ell + \frac{1}{2}$.  Therefore any line subbundle of $E[2u \posmod v]$ has slope at most $2\ell + 3$, and any rank two subbundle of $E[2u \posmod v]$ has slope at most $2\ell + \frac{3}{2}$.
    
Observe that $E[2u \posmod v]$ must be a two-cyclic modification of $N_C$.
Thus, by Lemma \ref{lem:black magic} either it is stable or its maximal destabilizing subbundle is a line bundle $F$ of degree $2\ell + 2$.  The intersection $F'$ of $F$ with $E \subset E[2u \posmod v]$ would have to be a line bundle of degree $2\ell$ or $2\ell + 1$, and so $\dim_k \Hom(F',E) \leq 2$.  As the isomorphism class of $F'$ cannot depend on $v$, this contradicts generality of $v \in \PP^4$. 
\end{proof}

\begin{lemma}\label{lem: U22 one-secant mod}
If $E$ has type $\Utype{2}{2}{\frac{1}{3}}$, then $E[2u \posmod v]$ has type $\Stype{1}$, $\Utype{1}{1}{\frac{2}{3}}$, or $\Utype{1}{2}{\frac{1}{6}}$.
\end{lemma}

\begin{proof}
Observe that $E[u \posmod v]$ must be semistable by generality of $v \in \PP^4$. Our claim about $E[2u \posmod v]$ follows immediately.
\end{proof}

\begin{remark}\label{rem: U22 restriction}
Suppose $E$ has type $\Utype{2}{2}{\frac{1}{3}}$ and $E[2u \posmod v]$ has type $\Utype{1}{1}{\frac{2}{3}}$. Write $Q$ for the minimal destabilizing quotient of $E$. Then the destabilizing line subbundle of $E[2u \posmod v]$ is isomorphic to $Q(2u)$.  Indeed, $Q$ is the only line bundle of degree $2\ell$ or $2\ell + 1$ with $\dim_k \Hom(Q, E) = 3$, and therefore the only such line bundle that can pick up the entire modification $[2u \posmod v]$.
\end{remark}

\noindent
We summarize our findings in the following proposition.

\begin{proposition}\label{prop: one secant}
Let $C \subset \PP^4$ be a general $BN$-curve of even degree and genus $g = 1$.  Let $E$ and $E_{n,m}$ be vector bundles of rank $3$ as in Proposition \ref{prop: types of bundles summary}; in particular, $E$ and $E_{n,m}$ are two-cyclic modifications of~$N_C$.  For $k \geq 1$ general pairs of points $(u_i,v_i) \in C \times \PP^4$, consider the modification
$$E_{n,m,k} = E_{n,m}[2u_1 \posmod v_1] \cdots [2u_k \posmod v_k].$$
The bundle $E_{n,m,k}$ has one of the \(7\) types $\Stype{0}$, $\Stype{1}$, $\Stype{2}$, $\Utype{1}{1}{\frac{2}{3}}$, $\Utype{1}{2}{\frac{1}{6}}$, $\Utype{2}{1}{\frac{1}{3}}$, $\Utype{2}{2}{\frac{1}{3}}$ described by Definition \ref{def-types_of_bundles}.
Moreover, if $k$ is even then $E_{n,m,k}$ cannot have type $\Utype{1}{2}{\frac{1}{6}}$ or $\Utype{2}{1}{\frac{1}{3}}$, and if \(k \geq 2\) then $E_{n,m,k}$ cannot have type \(\Utype{1}{1}{\frac{2}{3}}\).
\end{proposition}
\begin{proof}
When $k$ is even, $E_{n,m,k}$ is a two-cyclic modification of $N_C$ because there is a free action of $\mathbb{Z}/2\mathbb{Z}$ on the set of pairs $\{(u_1,v_1),\ldots ,(u_k,v_k)\}$.
    
Thus, we may apply Lemmas \ref{lem: S_0 one-secant mod} through \ref{lem: U22 one-secant mod} to find the possible types $E_{n,m,k}$.  The following table displays the possible transitions between $E_{n,m,k}$ and $E_{n,m,k+1}$, where the rows correspond to the possible types of $E_{n,m,k}$ and the columns correspond to possible types of $E_{n,m,k+1}$.

\begin{center}
{\renewcommand{\arraystretch}{1.2}
\begin{tabular}{c||c|c|c|c|c|c|c}
{} & $\Stype{0}$ & $\Stype{1}$ & $\Stype{2}$ & $\Utype{1}{1}{\frac{2}{3}}$ & $\Utype{1}{2}{\frac{1}{6}}$ & $\Utype{2}{1}{\frac{1}{3}}$ & $\Utype{2}{2}{\frac{1}{3}}$ \\ \hline\hline
$\Stype{0}$ & & & $\bullet$ & & & $\bullet$ & $\bullet$ \\ \hline 
$\Stype{1}$ & $\bullet$ & & & & & & \\ \hline
$\Stype{2}$ & & $\bullet$ & & & & & \\ \hline
$\Utype{1}{1}{\frac{2}{3}}$ & $\bullet$ & & & & & & \\ \hline
$\Utype{1}{2}{\frac{1}{6}}$ & $\bullet$ & & & & & & \\ \hline
$\Utype{2}{1}{\frac{1}{3}}$ & & $\bullet$ & & & & & \\ \hline
$\Utype{2}{2}{\frac{1}{3}}$ & & $\bullet$ & & $\bullet$ & $\bullet$ & & 
\end{tabular}}
\end{center}
This proves all but one desired fact: if $c_1(E_{n,m,k}) \equiv 1$ mod $3$ and $k \geq 2$, then $E_{n,m,k}$ cannot have type $\Utype{1}{1}{\frac{2}{3}}$.  To rule this out, observe that $E_{n,m,k-2}$ has type $\Stype{0}$.  If $E_{n,m,k}$ has type $\Utype{1}{1}{\frac{2}{3}}$, then $E_{n,m,k-1}$ must have type $\Utype{2}{2}{\frac{1}{3}}$.  In this case, let $\phi : E_{n,m,k-1}\to Q$ be the minimal destabilizing quotient of $E_{n,m,k-1}$.   As $\mu(Q) = \mu(E_{n,m,k-2})$, the restriction of $\phi$ to $E_{n,m,k-2} \subset E_{n,m,k-1}$ is also surjective, which limits the first Chern class of $Q$ to one of finitely many possibilities.  Therefore the isomorphism type of $Q$ cannot depend on $(u_{k-1}, v_{k-1})$.  But by Remark~\ref{rem: U22 restriction}, the maximal destabilizing subbundle $F \subset E_{n,m,k}$ must be $Q(2u_k)$.  This contradicts the invariance of $F$ under swapping $(u_{k-1},v_{k-1})$ with $(u_k, v_k)$.   
\end{proof}

\section{Proof of Theorem \ref{thm-mainstab}\label{sec:final}}

Throughout this section, given a modification $E$ of the normal bundle of a smooth curve $C \subset \PP^4$, for integers $n,m,k \geq 0$ we let $E_{n}, E_{n,m}$, and $E_{n,m,k}$ denote the modifications
\begin{align*}
E_{n} &= E[x_1 \biposmod  y_1][x_1 \posmod  2y_1][y_1 \posmod 2x_1]\cdots [x_n \biposmod  y_n][x_n \posmod  2y_n][y_n \posmod 2x_n]\\
E_{n,m} &= E_n[z_1 \biposmod  w_1]\cdots [z_m \biposmod  w_m]\\
E_{n,m,k} &= E_{n,m}[2u_1 \posmod v_1] \cdots [2u_k \posmod v_k]
\end{align*}
where the $x_i, y_i, z_i, w_i, u_i$ form a general collection of points in $C$ and the $v_i$ form a general collection of points in $\PP^4$.

\begin{remark}\label{rem: proving stability}
Let $C \subset \PP^4$ be a smooth BN-curve and let $C'$ be the nodal union of $C$ with $s$ general bisecant lines and $k$ general one-secant lines.  Suppose $E'$ is a vector bundle on $C'$ equipped with an isomorphism to $N_{C'}$ over an open subset of $C'$ containing each bisecant and one-secant line to $C$.  Write $E$ for the subbundle of $E'|_C$ corresponding to sections that do not smooth the nodes of $C'$, so that $E$ is naturally isomorphic to $N_C$ in a neighborhood of each node.  By Remarks \ref{rem: thanksgeoffandizzet} and \ref{rem:iterative extra juice}, to prove semistability (resp.\ stability) of $E'$, it is sufficient to prove the following statements about $E_{n,m,k}$ for all $n, m \geq 0$ such that $n + m = s$:
\begin{enumerate}
\item \label{linesubbound} Every line subbundle of $E_{n,m,k}$ has slope at most (resp.\ less than) $\mu(E_{n,m,k}) + \frac{1}{3}n + \frac{2}{3}m$;
\item \label{rktwosubbound}Every rank two subbundle of $E_{n,m,k}$ has slope at most (resp.\ less than) $\mu(E_{n,m,k}) +  \frac{1}{3}n + \frac{1}{6}m$.
\end{enumerate}
\end{remark}

\noindent
The following result is a consequence of Section \ref{sect: types of bundles}.

\begin{lemma}\label{lem: Enmk stability}
Adopt the notation of Remark \ref{rem: proving stability}.  Suppose $C \subset \PP^4$ is a general elliptic curve of even degree, $E$ is a two-cyclic modification of $N_C$ with type $\Stype{0}$, and $k \leq 2$.  Then $E_{n,m,k}$ satisfies the slope bounds of Remark \ref{rem: proving stability} required for semistability if
\[(n,m,k) \notin \{(0,1,0),(0,0,1)\}.\] 
Moreover  $E_{n,m,k}$ satisfies the slope bounds required for stability if
\[(n,m,k) \notin \{(0,0,0),(0,1,0),(0,0,1),(1,0,1), (0,1,1), (0,2,2)\}.\]
\end{lemma}
\begin{proof}
By Propositions \ref{prop: types of bundles summary} and \ref{prop: one secant}, $E_{n,m,k}$ has one of \(7\) possible types: $\Stype{0}$, $\Stype{1}$, $\Stype{2}$, $\Utype{1}{1}{\frac{2}{3}}$, $\Utype{1}{2}{\frac{1}{6}}$, $\Utype{2}{1}{\frac{1}{3}}$, $\Utype{2}{2}{\frac{1}{3}}$. This implies the bounds in Remark \ref{rem: proving stability} for \(m \geq 3\).

The $9$ cases \((m, k) \in \{0, 1, 2\} \times \{0, 1, 2\}\) are straightforward checks left to the reader. Note that as $n$ increases the bounds become more lenient so once they are satisfied they are satisfied forever.
\end{proof}

\begin{corollary}\label{cor: most cases}
For the following pairs $(g,d)$, the normal bundle of a general BN-curve of genus $g$ and degree $d$ in $\mathbb{P}^4$ is stable:
$$(g,d) \in \{(2,7), (2,9), (3,8), (3,10), (4,9), (4,10), (5,10), (5,11), (6,11), (7,12), (8,12), (9,13)\}.$$
\end{corollary}
\begin{proof}
Let $C \subset \PP^4$ be a general BN-curve of genus \(1\) and degree \(6\). To prove our claim, we will consider nodal unions of $C$ with other curves in $\PP^4$.

First consider the cases \((g, d) \in \{(8, 12), (9, 13)\}\).  Let $C'$ be the nodal union of $C$ with \(3\) general bisecant lines $l_i$ and the unique line $l$ meeting each $l_i$.  The restriction of $N_{C'}$ to any line in $C'$ is semistable by generality of the $l_i$.  Moreover, the bundle $E = N_{C'}|_C$ has type $\Stype{0}$ by Proposition \ref{prop: types of bundles summary}.  Let $D$ be the nodal union of $C'$ with $g-6$ general bisecant lines to $C$.  By Lemma \ref{lem: Enmk stability}, the bounds required by Remark \ref{rem: proving stability} are satisfied since $n+m = g - 6 \geq 2$ and $k = 0$.

In all other cases, let $D$ be the nodal union of $C$ with $g-1$ general bisecant lines and $k = d-g - 5$ general one-secant lines.  Let $E = N_C$.  By Lemma \ref{lem: Enmk stability}, the bounds required by Remark \ref{rem: proving stability} are satisfied, except for the bounds required for \(E_{0, 1, 0}\) and \(E_{0, 2, 2}\), as the other exceptions in Lemma \ref{lem: Enmk stability} do not appear for the desired \((g, d)\). 

Because $N_C$ is a direct sum of distinct line bundles by Theorem \ref{thm: normal bundle elliptic}, Remark \ref{rem: even mod distinct JHF} shows $E_{0,1,0}$ is stable.  Lemma \ref{lem: S2 one secant mod} implies $E_{0,1,1}$ is stable. Thus, for any line bundle $L$ of degree $\mu(E_{0,1,2})$, we have $\dim_k \Hom(L, E_{0,1,1}[u \posmod v]) = 0$ for any $u \in C$ and $v \in \PP^4$ general.  Therefore $\dim_k \Hom(L, E_{0,1,2}) \leq 1$ for any such line bundle $L$.  
By Remark \ref{rem: even mod distinct JHF}, this implies $E_{0,2,2}$ is stable.
\end{proof}

\noindent
To handle a few of the remaining cases, we will need the following lemma.

\begin{lemma}\label{lem: deg 4 elliptic curve modifications}
Let $C \subset \PP^3$ be an elliptic normal curve.  For $n, m \geq 0$ with $n \leq 4 + 2m$, let $p$, $y_1,\ldots , y_n$, $z_1, \ldots , z_m$, $w_1, \ldots , w_m$ be general points of $C$.  If $n \geq 2$ or $m \geq 2$, the bundle
$$N_C[y_1 + \cdots + y_n \posmod p][z_1 \biposmod w_1]\cdots [z_m \biposmod w_m]$$
is semistable and not the direct sum of two isomorphic line bundles.
\end{lemma}
\begin{proof}
Recall that $C$ is a complete intersection of two quadrics.  If $l_1, l_2$ are two general bisecant lines to $C$, then they lie in distinct quadrics containing $C$, even if both $l_i$ pass through a common point $p \in C$.  This demonstrates that
$$N_C[y_1 + y_2 \posmod p] \cong \OO_C(2)(y_1) \oplus \OO_C(2)(y_2) \ \text{and} \ N_C[z_1 \biposmod w_1][z_2 \biposmod w_2] \cong \OO_C(2)(z_1 + w_1) \oplus \OO_C(2)(z_2 + w_2).$$
Thus our claim is true when $(n,m) \in \{(0,2),(2,0)\}$.  We finish the proof by induction, showing that if the claim holds for $(n,m)$, then it holds for $(n+1, m)$; and if the claim holds for $(2n, m)$ then it holds for $(2n, m+1)$.  For this purpose, let $N_{n,m} = N_C[y_1 + \cdots + y_n \posmod p][z_1 \biposmod w_1]\cdots [z_m \biposmod w_m]$.

Suppose the claim holds for $(n,m)$.    By hypothesis, if $n$ is odd, then $N_{n,m}$ is stable.  In this case, $N_{n+1,m}$ must be semistable but not the direct sum of two isomorphic line bundles.  If instead $n$ is even, then $N_{n,m}$ is semistable but not the direct sum of two isomorphic line bundles.  Provided that $n < 4 + 2m$, the saturation of the pointing bundle $N_{C \to p}$ in $N_{n,m}$ cannot be a direct summand.  Thus, for a general choice of $y_{n+1}$, the resulting bundle $N_{n+1,m}$ is stable.
    
Now suppose the claim holds for $(2n,m)$.  Observe that the modification $[z_i \posmod w_i]$ is linearly general for general $w_i$.  Moreover, by hypothesis $N_{2n,m}$ is semistable and of integral slope.  Thus $N_{2n,m}[z_{m+1} \posmod w_{m+1}]$ is stable, and so $N_{2n,m+1}$ is semistable and not the direct sum of two isomorphic line bundles.
\end{proof}

\begin{lemma}\label{lem: extend to genus 8} 
The normal bundle of a general BN-curve of genus $g$ and degree $d$ in $\PP^4$ is stable in the cases $(g,d) \in \{(10, 13), (11, 14)\}$.
\end{lemma}
\begin{proof}
Let $C$ be a general BN-curve of genus 1 and degree 6.  Let $C'$ be the nodal union of $C$ with two general five-secant twisted cubics \(R_1\) and \(R_2\).  Let $E = N_{C'}|_C$.  Note that $E$ is a two-cyclic modification of $N_C$.  Indeed, for a given point $p \in C$, the set of \(5\)-secant twisted cubics  disjoint from $p$ and contained in a hyperplane through $p$ is parameterized by a rational variety.  As $g \geq 10$, let $D$ be the nodal union of $C'$ with $g - 9$ general bisecant lines to $C$.  It suffices to prove subsheaves of $E_{n,m}$ satisfy slope bounds given by Remark \ref{rem: proving stability} when $n + m = g - 9$.  
By Lemma \ref{lem: onion degeneration}, there is a degeneration of \(R_1\) to the union of three bisecant lines $\overline{p_1q_i}$ to $C$ passing through a single point $p_1 \in C$, and moreover, the resulting transformation at \(p_1\) is general. Degenerate \(R_2\) to the nodal union of a bisecant line $\overline{p_2q_4}$ and a conic $T$ meeting $C$ along $q_5, q_6, q_7 \in C$. This induces a specialization of \(E\) to
\[N_C[q_1 + q_2 + q_3 \posmod p_1][p_1 \posmod V][q_4 \biposmod p_2][q_5 + q_6 + q_7 \posmodalong T]\]
where \(V \subset N_C|_{p_1}\) is a general \(2\)-dimensional subspace. 
Then further limit \(p_1\) and \(p_2\) to a common point \(p\).  The modifications $[p_1 \posmod V]$ and $[p_2 \posmod q_4]$ specialize to the modification $[p \posmod V + q_4]$, which is equivalent to a twist up at $p$.  Thus, this induces a specialization of the above bundle to
\[E' = N_C(p)[q_1 + q_2 + q_3 + q_4 \posmod p][q_5 + q_6 + q_7 \posmodalong T].\]
By Remark~\ref{rem: exact sequence mod}, the bundle \(E'\) fits into the exact sequence 
$$0 \rightarrow N_{C \rightarrow p}(q_1 + q_2 + q_3 + q_4 + p) \rightarrow E' \rightarrow N_{\overline{C}}(2p)[q_5 + q_6 + q_7 \posmodalong \overline{T}] \rightarrow 0,$$
where $\overline{C}, \overline{T} \subset \PP^3$ are the images of $C$ and $T$ under projection from $p$.
We claim that $N_{\overline{C}}(2p)[q_5 + q_6 + q_7 \posmodalong \overline{T}]$ is either stable or the direct sum of two line bundles of degree $13$ and $14$.  Indeed, $\overline{T}$ is a general conic through the points $q_4, q_5, q_6, q_7 \in \overline{C}$.  It follows that the direction of the modification $[q_5 \posmodalong \overline{T}]$ is general.  Using Lemma \ref{lem: elliptic in P3}, we conclude that $N_{\overline{C}}(2p)[q_5 \posmodalong \overline{T}]$ is stable.  Making two more positive modifications, we see that $N_{\overline{C}}(2p)[q_5 + q_6 + q_7 \posmodalong \overline{T}]$ must either be stable or the direct sum of line bundles of degrees $13$ and $14$.

From Lemma \ref{lem:black magic}, it follows that $E$ has type $\Stype{1}$ or type $\Utype{1}{1}{\frac{2}{3}}$.  Either way, applying Lemma \ref{lem: S1 U11 U12 small mod} and then Lemma \ref{lem: S_0 big mod} shows $E_{0,1}$ and $E_{1,1}$ must be semistable.

To see the desired slope bounds for $E_{1,0}$ and $E_{0,2}$, specialize each \(R_i\) to the union of three bisecant lines $\overline{p_iq_{3i - 3 + j}}$ through a common point $p_i \in C$ using Lemma \ref{lem: onion degeneration}.  Specialize \(p_1\) and \(p_2\) to a common point \(p\). Consider the pointing bundle exact sequence for $p$:
\begin{align*}
0 \rightarrow N_{C \rightarrow p}(q_1 + \cdots + q_6 + p) \rightarrow E_{1,0} & \rightarrow N_{\overline{C}}(2p + x_1 + y_1)[x_1 \biposmod y_1][p \posmod \overline{V}] \rightarrow 0\\
0 \rightarrow N_{C \rightarrow p}(q_1 + \cdots + q_6 + p) \rightarrow E_{0,2} & \rightarrow N_{\overline{C}}(2p)[z_1  \biposmod w_1][z_2 \biposmod w_2][p \posmod \overline{V}] \rightarrow 0
\end{align*}
where the $x_i, y_i, z_i, w_i$ are general points in $C$ and $\overline{V} \subset N_{\overline{C}}|_p$ is a general linear subspace of dimension $1$.  We claim the last nonzero bundle in each sequence is stable.  Indeed, in the first case specialize \(p\) to \(x_1\), and in the second, specialize $p$ to $z_1$ and $w_2$ to $w_1$.  The two bundles specialize to $N_{\overline{C}}(4x_1 + y_1)[y_1 \posmod x_1]$ and $N_{\overline{C}}(3z_1 + w_1)[z_2 \posmod w_1]$, respectively.  These specializations are stable by Lemma \ref{lem: elliptic in P3}.  We conclude that $E_{1,0}$ and $E_{0,2}$ satisfy the desired slope bounds from Remark \ref{rem: proving stability}.

Lastly we show that subsheaves of $E_{2,0}$ satisfy the desired slope bounds.  Specialize $C$ to the nodal union $C''\cup l$ of an elliptic normal curve $C''$ and a one-secant line $l = \overline{uv}$ such that $(u,v) \in C'' \times \PP^4$ is general, and such that the points of incidence of both \(R_i\) specialize onto \(C''\).
Write \(E'' = N_{C''}[2u \posmod v][\posmodalong R_1][\posmodalong R_2]\).
Applying Lemma~\ref{lem: pull off 1-secant}, it suffices to show that \(E''_{2,0}\) satisfies the slope bounds of Remark~\ref{rem: proving stability}.

As before, specialize each \(R_i\) to the union of three bisecant lines $\overline{p_iq_{3i - 3 + j}}$ through a common point $p_i \in C$ using Lemma \ref{lem: onion degeneration},
and then specialize \(p_1\) and \(p_2\) to a common point \(p\).
Specialize $v$ to $p$ as well, and consider the pointing bundle exact sequence:
\[0 \rightarrow N_{C'' \rightarrow p}(q_1 + \cdots + q_6 + p + 2u) \rightarrow E''_{2,0} \rightarrow N_{\overline{C}''}(2p + x_1 + x_2 + y_1 + y_2)[x_1  \biposmod y_1][x_2 \biposmod y_2][p \posmod \overline{V}] \rightarrow 0\]
where the $x_i, y_i$ are general points in $C''$ and $\overline{V} \subset N_{\overline{C}''}|_p$ is a general linear subspace of dimension \(1\).  By Lemma \ref{lem: deg 4 elliptic curve modifications}, the
quotient is stable and so $E_{2,0}$ satisfies the slope bounds from Remark \ref{rem: proving stability}.
\end{proof}

\begin{lemma}
The normal bundle of a general BN-curve of genus 2 and degree 8 in $\PP^4$ is stable. 
\end{lemma}

\begin{proof}
Consider the nodal union $D$ of an elliptic curve $C \subset \PP^4$ of degree 6 and a conic $R$ meeting $C$ at two points $p,q$.  Let $x \in \PP^4$ be the intersection of $T_p R$ and $T_q R$.    Let $\Delta \subset \PP^4$ be the linear span of $R$ and $\nu : C \sqcup R \rightarrow D$ be the normalization.  Note that 
$$N_D|_R \cong \OO(3) \oplus \OO(3) \oplus \OO(4)$$
while $N_D|_C \cong N_C[p + q \posmod x]$.  Since $N_C[p \posmod x]$ is stable by \cite[Proposition~4.5]{cs23}, every sub line bundle of $N_C[p + q \posmod x]$ has degree at most \(11\); hence, \(N_D|_C\) is stable by Lemma \ref{lem:black magic}.  This already shows that \(N_D\) is semistable. To establish stability, assume for the sake of contradiction that $F \subset \nu^* N_D$ is a proper subbundle of adjusted slope $\mu(N_D) = 14$.  Set $F_R = F|_R$ and $F_C = F|_C$.

First suppose $\rk(F) = 1$.  As $\mu(F_R) \leq 4$ and $\mu(F_C) \leq 10$, we must have equality everywhere and $F$ must be the pullback of a subbundle under $\nu$.  In particular $F_R|_p$ and $F_R|_q$ must be the restriction of $N_{R/\Delta} \subset N_R \subset N_D|_R$ to each respective point.  This places two restrictions on $F_C$: (1) $F_C \subset N_C \subset N_D|_C$ and (2) $F_C|_p$ is contained inside the span of $T_{\Delta}|_p \subset N_C|_p$.  However, by Theorem \ref{thm: normal bundle elliptic}, $N_C$ is a direct sum of three distinct line bundles of degree \(10\).  As we may choose $\Delta$ such that $T_{\Delta}|_p$ does not contain the fiber of any direct summand of $N_C$ over $p$, this implies $\mu(F_C)<10$.  Hence, $\rk(F) \neq 1$.

Next suppose $\rk(F) = 2$.  As $\mu(F_R) \leq 3\frac{1}{2}$ and $\mu(F_C) \leq 10\frac{1}{2}$, we must have equality everywhere and $F$ must be the pullback of a subbundle under $\nu$.  It follows that $F_R|_p$ and $F_R|_q$ must contain the restriction of $N_{R/\Delta} \subset N_R \subset N_D|_R$ to each respective point.  Hence the intersection $F'$ of $F_C$ with $N_C[p \negmod T_{\Delta}|_p][q \negmod T_{\Delta}|_q]$
has slope at least $9\frac{1}{2}$.  However, by the preceding paragraph, $N_C[p \negmod T_{\Delta}|_p]$ has no subbundles of slope \(10\).  Hence $F'$ must have slope exactly \(9 \frac{1}{2}\) and be the maximal destabilizing subbundle of $N_C[p \negmod T_{\Delta}|_p][q \negmod T_{\Delta}|_q]$.  This contradicts Lemma \ref{lem:black magic}.
\end{proof}

\begin{lemma}
The normal bundle of a general BN-curve of genus $15$ and degree $16$ in $\PP^4$ is stable.
\end{lemma}

\begin{proof}
Consider the nodal union $D$ of a general BN-curve $C \subset \PP^4$ of genus 3 and degree 7 with three 5-secant twisted cubics $R_i$.  Recall that there are three points $s_1, s_2, s_3 \in C$, the intersection of $C$ with its unique trisecant line, such that
$$N_C \cong \bigoplus_i \OO_C(2)(-s_i).$$
It suffices to prove the stability of $N = N_D|_C = N_C[\posmodalong R_1][\posmodalong R_2][\posmodalong R_3]$.  To do this, we will consider the limit of $N$ under a specialization of $D$. 
    
Specialize the $R_i$ so that each curve lies in the same hyperplane $H$.  We may assume $R_1$ and $R_2$ meet $C$ along five common points $q_i \in C \cap H$, $1 \leq i \leq 5$, while $R_3$ meets $C$ along $q_i$ for $i > 1$ and a distinct point $q_6 \in  C\cap H$. 
Because we may assume the tangent directions of $R_i$ are linearly independent at  $q_2$, $q_3$, $q_4$, and $q_5$, the restricted normal bundle $N$ specializes to the modified bundle 
$$N_C(q_2 + q_3 + q_4 + q_5)[q_1 \posmodalong R_1][q_1 \posmodalong R_2][q_6\posmodalong R_3].$$
We will show 
$N' = N_C[q_1 \posmodalong R_1][q_1 \posmodalong R_2][q_6\posmodalong R_3]$ is stable using generality of $T_{R_1}|_{q_1} + T_{R_2}|_{q_1}$ and $T_{R_3}|_{q_6}$.
    
First suppose $N'$ has a subbundle $F' \subset N'$ of slope at least $\mu(N') = 14$.  Let $F \subset N_C$ be the intersection of $F'$ with \(N_C \subset N'\).  If $\mu(F) \geq 13$, then \(\mu(F) = 13\) and $F$ must be a direct summand of $N_C$. Since $N_C$ is a direct sum of three distinct line bundles, generality of $T_{R_1}|_{q_1} + T_{R_2}|_{q_1}$ and $T_{R_3}|_{q_6}$ ensures both modifications are transverse to the image of $F$.  Thus $\mu(F') < 14$, a contradiction. We may therefore suppose \(\mu(F) < 13\).
    
Suppose $F$ has rank one.  If $\mu(F) < 12$, then $\mu(F') < 14$. We may therefore suppose $\mu(F) = 12$.
Since the twisted cubics $R_i$ vary in a rational family while $q_1,\ldots, q_6$ remain fixed, the isomorphism class of $F$ does not depend on the linear subspaces $T_{R_1}|_{q_1} + T_{R_2}|_{q_1}$ and $T_{R_3}|_{q_6}$.  
Since \(\dim_k \Hom(F, N_C) \leq 3\), there is only a two-dimensional family of possibilities for \(F\). Thus \(F\) cannot pick up both modifications.
    
Suppose $F$ has rank two.  If $\mu(F) < 12\frac{1}{2}$, then $\mu(F') < 14$. We may therefore suppose $\mu(F) = 12\frac{1}{2}$.
The cokernel \(N_C \to K\) of \(F \to N_C\) completely determines \(F\) and \(\mu(K) = 14\).
As before, the isomorphism class of $K$ does not depend on the choice of the $R_i$.
Since \(\dim_k \Hom(N_C, K) \leq 3\), there is only a two-dimensional family of possibilities for \(F\). Thus \(F\) cannot contain both modifications.
\end{proof}

\begin{lemma}\label{lem: genus 10 degree 12}
The normal bundle of a general BN-curve of genus $10$ and degree $12$ in $\PP^4$ is stable.
\end{lemma}

\begin{proof}
Consider the nodal union $D$ of an elliptic curve $C$ of degree 5 with one bisecant line $l$ and two \(5\)-secant twisted cubics $R_1, R_2$. There is a unique quadric $Q$ containing $C \cup R_1 \cup R_2$. This is because there are no reducible quadrics containing $C$, and if we specialize $R_1$ and $R_2$ to lie in the same hyperplane, no two quadrics can contain their union, which is a degree six curve.  Because the secant variety to \(C\) is a quintic hypersurface, we may suppose $l$ is not contained in $Q$.  This implies $l$ is transverse to $Q$.  Consider the exact sequence
$$0 \rightarrow N_{C \cup R_1 \cup R_2 / Q} \rightarrow N_{C \cup R_1 \cup R_2} \rightarrow N_{Q/\PP^4}|_{C \cup R_1 \cup R_2} \rightarrow 0.$$
We claim $N_{C \cup R_1 \cup R_2 / Q}$ is stable. 
Indeed, $N_{C \cup R_1 \cup R_2/Q}|_{R_i} \cong \OO(6)^{\oplus 2}$, so it suffices to prove the stability of $N_{C \cup R_1 \cup R_2/Q}|_C$.  As in the proof of Lemma \ref{lem:black magic}, multiplication by a five torsion point on $C$ fixes $N_{C \cup R_1 \cup R_2/Q}|_C$, and so the first Chern class of a destabilizing subbundle must be divisible by 5.  However, we will show in the following three paragraphs that $C \cup R_1$ is a complete intersection of quadrics, so that $N_{C \cup R_1/Q}|_C \cong \OO_C(2)^{\oplus 2}$.
Assuming this, the only quadric containing $C \cup R_1$ that is tangent to $R_2$ at each point $R_2 \cap C$ is $Q$.  Hence, $N_{C \cup R_1 \cup R_2/Q}|_C$ cannot have a subbundle of slope at least $15$, so $N_{C \cup R_1 \cup R_2/Q}|_C$ is stable as desired.  

To show $C \cup R_1$ is a complete intersection of quadrics, first observe that it lies in a net of quadrics spanned by $Q$ and two other quadrics $Q_1, Q_2$.  If  $Q \cap Q_1 \cap Q_2$ has dimension one, then this intersection must be $C \cup R_1$ for degree reasons.  Hence, we may suppose the base locus of this net of quadrics contains a surface $S \subset \PP^4$ that contains $C$.  In this case, $Q_1 \cap Q_2$ must be reducible.  Since $C$ is nondegenerate, it does not lie in a surface of degree less than $3$.  Hence, $S$ must be a nondegenerate cubic surface.  Since $R_1$ does not lie in a two-dimensional plane, it follows that $R_1 \subset S$ as well.

Up to automorphisms of $\PP^4$, there are only two possibilities for $S$: it is either the cone over a twisted cubic or a rational normal scroll of degree \(3\).  However, $S$ cannot be the cone over a twisted cubic, as $C$ cannot map to a twisted cubic upon projection from a point for degree reasons.  Therefore $S$ is a cubic scroll.  

There is a line $L \subset S$ such that projection from $L$ maps $S$ onto a conic in $\PP^2$.  Again for degree reasons, this projection map restricts to a degree two map $C \to \PP^1$.  The cubic scroll $S$ is determined by this map.  Hence, there is only a one-parameter family of cubic scrolls containing $C$, parameterized by \(\operatorname{Pic}^2 (C)\).  As $R_1$ is the intersection of a cubic scroll containing $C$ with a hyperplane $H$, upon fixing $C$ and $H$ there is only a one-dimensional family of such intersections.  However, there is a two-dimensional family of twisted cubics passing through the five intersection points $C \cap H$. The general \(5\)-secant twisted cubic is therefore not such an intersection, completing our proof that $C \cup R_1$ is a complete intersection of quadrics.  

Our final goal is to prove stability of $N_D$ using Lemma \ref{lem:extra juice}.  Let $x,y \in C \cap l$ be the two points of attachment in $D$.  Since $l$ meets $Q$ transversely, the restriction to $C$ of the normal bundle $N_D$ fits into an exact sequence
$$0 \rightarrow N_{C \cup R_1 \cup R_2 / Q}|_{C} \rightarrow N_D|_{C} \rightarrow N_{Q/\PP^4}|_{C}(x + y) \rightarrow 0.$$
Note that $\mu(N_{C \cup R_1 \cup R_2 / Q}|_{C}) = 12\frac{1}{2}$ and $\mu(N_{Q/\PP^4}|_{C}(x + y)) = 12$.  In particular, $N_D|_C$ has no subbundles of slope $13 = \mu(N_D|_C) + \frac{2}{3}$ or greater.  Furthermore, every rank two subbundle of $N_D|_C$ that smooths at least one of the nodes $x$ or $y$ has slope less than $12\frac{1}{2}$.  By Lemma \ref{lem:extra juice}, it therefore suffices to show the stability of
$$N = N_D|_C[x \posmod 2y][y \posmod 2x] = N_C[x \biposmod y][x \posmod 2y][y \posmod 2x][\posmodalong R_1][\posmodalong R_2].$$

Specialize each \(R_i\) to the union of three bisecant lines \(\overline{p_i q_{3i - 3 + j}}\) through a common point \(p_i\) using Lemma~\ref{lem: onion degeneration}. Then specialize \(p_1\) and \(p_2\) to a common point \(p\).  Observe that \[\OO_C(q_1 + q_2 + q_3 + 2p) \cong \OO_C(q_4 + q_5 + q_6 + 2p) \cong \OO_C(1).\]  The resulting specialization \(N' = N_C(p)[x \biposmod y][x\posmod  2y][y \posmod  2x][q_1 + \cdots + q_6 \posmod  p][p \posmod  V]\) of $N$ fits into the exact sequence
\[0 \to N_{C \to p}(p + q_1 + \cdots + q_6) \to N' \to N_{\overline{C}}(2p + x + y)[x \biposmod y][p \posmod  \overline{V}] \to 0,\]
where $\overline{C} \subset \PP^3$ is the projection of $C$ from $p$ and \([p \posmod \overline{V}]\) is a general modification.
Observe that we have $N_{\overline{C}}(2p + x + y)[x \posmod y][p \posmod  v] \cong \OO_C(2)(2x + 2y) \oplus \OO_C(2)(x + y + p)$, while $N_{C \rightarrow p}(p + q_1 + \cdots + q_6) \cong \OO_C(3)(-p)$.  Since each $R_i$ deforms in a rational family, the first Chern class of the maximal destabilizing subbundle of $N$ cannot depend on the choice of $R_i$.  Varying $p$, we see that if $N$ is unstable, it must have a destabilizing line bundle with first Chern class $\OO_C(2)(2x + 2y)$.

Suppose instead that 
$R_1$ specializes as before while $R_2$ remains general, and then specialize $p_1$ to $x$.  Because the space of possible \(R_1\) is rational, if $N$ is unstable then there must be a nontrivial map from $\OO_C(2)(2x + 2y)$ to the resulting specialization
$N' = N_C(x)[x\posmod  2y][y \posmod  2x][q_1 + q_2 + q_3 + y \posmod  x][\posmodalong R_2]$ of \(N\).  This specialization fits into the exact sequence
\[0 \to N_{C \to x}(x + 2y + q_1 + q_2 + q_3) \to N' \to N_{\overline{C}}(3x)[y \posmod  x][\posmodalong \overline{R}_2] \to 0,\]
where $\overline{C} \cup \overline{R}_2 \subset \PP^3$ is the projection of $C \cup R_2$ from $x$.  

It thus suffices to show that $N_{\overline{C}}(3x)[y \posmod  x][\posmodalong \overline{R}_2]$ is semistable and does not have a JH-factor with first Chern class $\OO_C(2)(2x + 2y)$.  We may suppose $\overline{R}_2$ meets $\overline{C}$ along any five points $q_4, \ldots, q_8$ whose sum is $\OO_C(1)$.  Specialize $\overline{R}_2$ to the union of a line through $q_4, q_5$ and a conic through $q_6, q_7, q_8$ meeting the line $\overline{q_4q_5}$.  We may assume $\overline{q_4q_5}$ meets $\overline{xy} \subset \PP^3$, so that $q_4 + q_5 + y + 2x = \OO_C(1)$.  In particular, $q_6 + q_7 + q_8 = y + 2x \neq x + 2y$.  Since $\overline{q_4q_5}$ and $\overline{xy}$ are contained in the same quadric containing $\overline{C}$, the bundle $N_{\overline{C}}(3x)[y \posmod  x][\posmodalong \overline{R}_2]$ splits as a direct sum $\OO_C(2)(x + y + q_4 + q_5) \oplus \OO_C(2)(x + q_6 + q_7 + q_8)$.  For general $x,y$, the quadric containing $\overline{xy} \cup \overline{C}$ is smooth, whence $q_4 + q_5 \neq x + y$.
\end{proof}

\begin{lemma}
The normal bundle of a general BN-curve of genus $g$ and degree $d$ in $\PP^4$ is stable in the cases $(g,d) \in \{(11, 13), (12, 14), (13,15)\}$.
\end{lemma}
\begin{proof}
Let $C \subset \PP^4$ be a general elliptic normal curve.  Consider the nodal union $D$ of $C$ with two general five-secant twisted cubics $R_1, R_2$, and $g - 9$ general bisecant lines.  We will prove $N_D$ is stable.
Let $E$ be the restriction to $C$ of the normal bundle of $C \cup R_1 \cup R_2$.  Consider the modification
$$E_{n,m} = E[x_1 \biposmod y_1][x_1 \posmod  2y_1][y_1 \posmod 2x_1]\cdots [x_n \biposmod  y_n][x_n \posmod  2y_n][y_n \posmod 2x_n][z_1 \biposmod  w_1]\cdots [z_m \biposmod  w_m]$$ 
where the $x_i,y_i,z_i,w_i$ are general points in $C$.  It suffices to prove that subbundles of $E_{n,m}$ satisfy the slope bounds of Remark \ref{rem: proving stability} for all $2 \leq n + m \leq 4$.  
Note that $\mu(E_{n,m}) = 11\frac{2}{3} + \frac{2}{3}m + 2n$.  We will show \ref{rem: proving stability}\eqref{linesubbound} and \ref{rem: proving stability}\eqref{rktwosubbound} hold for fixed values of $n$ and all values of $m$ simultaneously.  To do so, we specialize the twisted cubics to reducible curves and consider the corresponding limits $E'$ of $E$ and $E_{n,m}'$ of $E_{n,m}$ as modifications of $N_C$.

First suppose \(n = 0\).
Specialize each \(R_i\) to the union of three bisecant lines \(\overline{p_i q_{3i - 3 + j}}\) through a common point \(p_i\) using Lemma~\ref{lem: onion degeneration}, and then specialize \(p_1\) and \(p_2\) to a common point \(p\).
The resulting specialization $E_{0,m}'$ of $E_{0,m}$ fits into the pointing bundle sequence
\[0 \to N_{C \to p}(q_1 + \cdots + q_6 + p) \to E_{0, m}' \to N_{\overline{C}}(2p)[z_1 \biposmod w_1]\cdots [z_m \biposmod w_m][p \posmod \overline{V}] \to 0.\]
For $2 \leq m \leq 4$, Lemma \ref{lem: deg 4 elliptic curve modifications} demonstrates that the quotient of $E_{0, m}'$ is a stable rank two bundle of slope $10\frac{1}{2} + m$, while the subbundle is a line bundle of degree $14$.  This proves \ref{rem: proving stability}\eqref{linesubbound} and \ref{rem: proving stability}\eqref{rktwosubbound} for $n = 0$.

Suppose $n = 1$ instead  and specialize the twisted cubics as before.  When $m = 1$, the resulting specialization $E_{1,1}'$ of $E_{1,1}$ fits into an exact sequence
$$0 \to N_{C \rightarrow p}(q_1 + \cdots + q_6 + p) \to E_{1,1}' \to N_{\overline{C}}(2p + x_1 + y_1)[x_1 \biposmod  y_1][z_1 \biposmod w_1][p \posmod \overline{V}]\to 0.$$
Again, Lemma \ref{lem: deg 4 elliptic curve modifications} demonstrates that the quotient is stable.  This proves \ref{rem: proving stability}\eqref{linesubbound} and \ref{rem: proving stability}\eqref{rktwosubbound} for $(n,m) = (1,1)$.  If $m \geq 2$ instead, specialize $z_1$ to $p$ as well.  The resulting specialization $E_{1,m}''$ of $E_{1,m}'$ sits in an exact sequence
$$0 \to N_{C \rightarrow p}(q_1 + \cdots + q_6 + p + w_1) \to E_{1,m}'' \to N_{\overline{C}}(3p + x_1 + y_1)[x_1 \biposmod  y_1][z_2 \biposmod w_2] \cdots [z_m \biposmod w_m] \to 0.$$
By Lemma \ref{lem: deg 4 elliptic curve modifications} the quotient is semistable.  This proves \ref{rem: proving stability}\eqref{linesubbound} and \ref{rem: proving stability}\eqref{rktwosubbound} for $n = 1$.  

Suppose $n = 3$ and specialize the twisted cubics as before once more.  Specialize $x_1$ to $p$, and, if $m = 1$, then specialize $z_1$ to $p$ as well.  The resulting specializations $E_{3,0}'$ and $E_{3,1}'$ fit into exact sequences
{\small \begin{gather*}
0 \to N_{C \rightarrow p}(q_1 + \cdots + q_6 + 2p + 2y_1) \to E_{3,0}' \to N_{\overline{C}}(3p + x_2 + y_2 + x_3 + y_3)[p \biposmod y_1][x_2 \biposmod y_2][x_3 \biposmod y_3] \to 0; \\
0 \to N_{C \rightarrow p}(q_1 + \cdots + q_6 + 2p + 2y_1 + w_1) \to E_{3,1}' \to N_{\overline{C}}(4p + x_2 + y_2 + x_3 + y_3)[y_1 \posmod p][x_2 \biposmod y_2][x_3 \biposmod y_3] \to 0.
\end{gather*}}
By Lemma \ref{lem: deg 4 elliptic curve modifications} the quotients are semistable.  This proves \ref{rem: proving stability}\eqref{linesubbound} and \ref{rem: proving stability}\eqref{rktwosubbound} for $n = 3$.

When $n = 2$, specialize just one of the twisted cubics $R_1$ to a union of lines $\overline{pq_i}$ through $p \in C$.  Specialize $x_1$ and $x_2$ to $p$ as well.
We have an exact sequence 
\[0 \to N_{C \rightarrow p}(q_1 + q_2 + q_3 + 2y_1 + 2y_2 + 2p) \rightarrow E_{2,m}' \to N_{\overline{C}}(4p)[y_1 + y_2 \posmod p][\posmodalong \overline{R}_2][z_1 \biposmod w_1]\cdots [z_m \biposmod w_m].\]
We will show the quotient in the above sequence is stable.  To do so, let $q_4, \ldots, q_8 \in C$ be the five intersection points with $R_2$.  Specialize $y_1$ to $q_8$, so that the quotient specializes to 
$$N_{\overline{C}}(4p + q_8)[y_2 \posmod p][q_4 + q_5 + q_6 + q_7 \posmodalong \overline{R}_2][z_1 \biposmod w_1]\cdots [z_m \biposmod w_m].$$
We may specialize $R_2$ to a union of lines $\overline{q_4q_5}$, $\overline{q_6q_7}$, and the unique line $L$ that contains $q_8$ and passes through both $\overline{q_4q_5}$ and $\overline{q_6q_7}$.  This induces a specialization of the above quotient to 
$$N_{\overline{C}}(4p + q_8)[y_2 \posmod p][q_4 \biposmod q_5][q_6 \biposmod q_7][z_1 \biposmod w_1]\cdots [z_m \biposmod w_m].$$
This bundle is stable by Lemma \ref{lem: deg 4 elliptic curve modifications}.  We conclude that \ref{rem: proving stability}\eqref{linesubbound} and \ref{rem: proving stability}\eqref{rktwosubbound} hold for $n = 2$. 

Finally, when $n = 4$, we specialize all four $x_i$ to $p$.  The resulting specialization of $E_{4,0}$ fits into an exact sequence
$$0 \to N_{C \rightarrow p}(2y_1 + \cdots + 2y_4 + 4p) \rightarrow E_{4,0}' \to N_{\overline{C}}(5p)[y_1 + \cdots + y_4 \posmod p][\posmodalong \overline{R}_1][\posmodalong \overline{R}_2].$$ 
By Lemma \ref{lem: deg 4 elliptic curve modifications} $N_{\overline{C}}(5p)[y_1 + \cdots + y_4 \posmod p]$ is semistable.  Upon specializing $C \cap R_1$ to $C \cap R_2$, the quotient specializes to $N_{\overline{C}}(5p + \overline{C} \cap \overline{R}_2)[y_1 + \cdots + y_4 \posmod p]$.  This proves \ref{rem: proving stability}\eqref{linesubbound} and \ref{rem: proving stability}\eqref{rktwosubbound} for $n = 4$. 
\end{proof}

\begin{lemma}
The normal bundle of a general BN-curve of genus $g$ and degree $d$ in $\PP^4$ is stable in the cases $(g,d) \in \{(4, 8), (5, 9), (6,10)\}$.
\end{lemma}

\begin{proof}
Consider the nodal union $D$ of an elliptic curve $C \subset \PP^4$ of degree $5$ together with $g - 1$ general bisecant lines.  We will show $N_D$ is stable by using Remark \ref{rem: proving stability} and studying specializations of $D$ where various bisecant lines collide to meet at a point $p \in C$.  To fix notation, let $E = N_C$ and 
$$E_{n,m} = E[x_1 \biposmod  y_1][x_1 \posmod  2y_1][y_1 \posmod 2x_1]\cdots [x_n \biposmod  y_n][x_n \posmod  2y_n][y_n \posmod 2x_n][z_1 \biposmod  w_1]\cdots [z_m \biposmod  w_m].$$
    
By Remark \ref{rem: proving stability}, to show $N_D$ is stable it suffices to prove that subbundles of $E_{n,m}$ satisfy the slope bounds required by Remark \ref{rem: proving stability} for $3 \leq n + m \leq 5$.  Note that $\mu(E_{n,m}) = 8 \frac{1}{3} + 2n + \frac{2}{3}m$.  We analyze these bundles according to the value of $n$.

First, suppose $n = 0$.  Specialize the $z_i$ to $p$ one at a time. The resulting specialization $E_{0, m}'$ of $E_{0, m}$ fits into an exact sequence
$$0 \rightarrow N_{C\rightarrow p}(p + w_1 + \cdots + w_m) \rightarrow E'_{0,m} \rightarrow N_{\overline{C}}(2p)[p \posmod w_4 + \cdots + w_m] \rightarrow 0.$$
The subbundle has slope $8 + m$. If $m \neq 4$, then by Lemma \ref{lem: deg 4 elliptic curve modifications}, the quotient is semistable; otherwise, if \(m = 4\), then the quotient is the direct sum of line bundles of degrees $10$ and $11$.  Hence, subbundles of $E_{0,m}$ satisfy the slope bounds of Remark \ref{rem: proving stability}.

Similarly, suppose $n = 1$, and specialize \(x_1, z_1, z_2, \ldots, z_m\) to \(p\) one at a time in this order. The resulting specializations \(E_{1, m}'\) of \(E_{1, m}\) fit into the exact sequences
\begin{align*}
0 \to N_{C \to p}(p + 2y_1 + w_1 + w_2) \to E_{1, 2}' &\to N_{\overline{C}}(3p)[y_1 \posmod p] \to 0;\\
0 \to N_{C \to p}(2p + 2y_1 + w_1 + w_2 + w_3) \to E_{1, 3}' &\to N_{\overline{C}}(3p)[y_1 \posmod p] \to 0;\\
0 \to N_{C \to p}(2p + 2y_1 + w_1 + w_2 + w_3 + w_4) \to E_{1, 4}' &\to N_{\overline{C}}(3p)[y_1 \posmod p][p \posmod w_4] \to 0.
\end{align*}    
The subbundles above have slopes \(12\), \(14\), and \(15\) respectively. Moreover, the quotients are direct sums of line bundles of degree at most \(12\).  More precisely, \(N_{\overline{C}}(3p)[y_1 \posmod p] \simeq \OO_{\overline{C}}(2)(3p) \oplus \OO_{\overline{C}}(2)(3p + y_1)\) and \(N_{\overline{C}}(3p)[y_1 \posmod p][p \posmod w_4] \simeq \OO_{\overline{C}}(2)(3p + y_1) \oplus \OO_{\overline{C}}(2)(4p)\). Hence subbundles of \(E_{1, m}\) satisfy the slope bounds of Remark \ref{rem: proving stability}.

Suppose $n = 2$.  Specialize $x_1, x_2$ to $p$, and then specialize $z_1,z_2$ to $p$ if $m \geq 2$.  The resulting specializations $E_{2,m}'$ of $E_{2,m}$ fit into exact sequences
\begin{align*}
0 \to N_{C \to p}(2p + 2y_1 + 2y_2) \to E_{2, 1}' &\to N_{\overline{C}}(3p)[y_1 + y _2 \posmod p][z_1 \biposmod w_1] \to 0;\\
0 \to N_{C \to p}(2p + 2y_1 + 2y_2 + w_1 + w_2) \to E_{2, 2}' &\to N_{\overline{C}}(4p)[y_1 + y_2 \posmod p] \to 0;\\
0 \to N_{C \to p}(2p + 2y_1 + 2y_2 + w_1 + w_2) \to E_{2, 3}' &\to N_{\overline{C}}(4p)[y_1 + y_2 \posmod p][z_3 \biposmod w_3] \to 0.
\end{align*}
The subbundles above have slopes $13$, $15$, and $15$, respectively, while the quotients are semistable by Lemma \ref{lem: deg 4 elliptic curve modifications}.  Hence subbundles of \(E_{2, m}\) satisfy the slope bounds of Remark \ref{rem: proving stability}.

Suppose $n = 3$.  Specialize $x_1,x_2,x_3$ to a common point $p$.  The resulting specializations $E_{3,m}'$ of $E_{3,m}$ fit into exact sequences
\begin{align*}
0 \to N_{C \to p}(2p + 2y_1 + 2y_2 + 2y_3) \to E_{3,0}' &\to N_{\overline{C}}(4p)[y_1 + y_2 + y_3 \posmod p][p \posmod y_3] \to 0;\\
0 \to N_{C \to p}(2p + 2y_1 + 2y_2 + 2y_3) \to E_{3,1}' &\to N_{\overline{C}}(4p)[y_1 + y_2 + y_3 \posmod p][z_1 \biposmod w_1][p \posmod y_3] \to 0;\\
0 \to N_{C \to p}(2p + 2y_1 + 2y_2 + 2y_3) \to E_{3,2}' &\to N_{\overline{C}}(4p)[y_1 + y_2 + y_3 \posmod p][z_1 \biposmod w_1][z_2 \biposmod w_2][p \posmod y_3] \to 0.
\end{align*}
The subbundles above have slope \(15\).  By Lemma \ref{lem: deg 4 elliptic curve modifications}, the quotients are the modification by $[p \posmod y_3]$ of a stable bundle, and are thus semistable.  Hence subbundles of \(E_{3, m}\) satisfy the slope bounds of Remark \ref{rem: proving stability}. 

Suppose $n = 4$.  Specialize $x_1, x_2, x_3$ to a common point $p$ and, if $m = 1$, specialize $z_1$ to $p$ as well.  The resulting specializations $E_{4,m}'$ of $E_{4,m}$ fit into exact sequences
\begin{align*}
0 \to N_{C \to p}(2p + 2y_1 + 2y_2 + 2y_3) \to E_{4,0}' &\to N_{\overline{C}}(4p + x_4 + y_4)[y_1 + y_2 + y_3 \posmod p][x_4 \biposmod y_4][p \posmod y_3] \to 0;\\
0 \to N_{C \to p}(3p + 2y_1 + 2y_2 + 2y_3 + w_1) \to E_{4,1}' &\to N_{\overline{C}}(4p + x_4 + y_4)[y_1 + y_2 + y_3 \posmod p][x_4 \biposmod y_4][p \posmod y_3] \to 0.
\end{align*}
The subbundles have slopes 15 and 17.  As before, Lemma \ref{lem: deg 4 elliptic curve modifications} shows the quotients are semistable.  Hence subbundles of \(E_{4, m}\) satisfy the slope bounds of Remark \ref{rem: proving stability}

Lastly, suppose $n = 5$.  Specialize $x_1, \ldots , x_4$ to a common point $p$.  The resulting specialization $E_{5,0}'$ of $E_{5,0}$ fits into an exact sequence
\[0 \to N_{C \to p}(4p + 2y_1 + 2y_2 + 2y_3 + 2y_4) \to E_{5,0}' \to N_{\overline{C}}(5p + x_5 + y_5)[y_1 + y_2 + y_3 + y_4 \posmod p][x_5 \biposmod y_5] \to 0.\]
The subbundle has slope \(19\) while the quotient is semistable by Lemma \ref{lem: deg 4 elliptic curve modifications}.  Hence subbundles of $E_{5,0}$ satisfy the slope bounds of Remark \ref{rem: proving stability}.  This finishes our proof. 
\end{proof}

\begin{lemma}\label{lem: last lem}
The normal bundle of a general BN-curve of genus $8$ and degree $11$ in $\PP^4$ is stable.
\end{lemma}
\begin{proof}
Consider the nodal union $D$ of an elliptic curve $C$ of degree $5$ with three general bisecant lines and one 5-secant twisted cubic $R$.  
We will show $N_D$ is stable by using Remark \ref{rem: proving stability} and studying specializations of $D$ where various bisecants collide to meet at a point $p \in C$.  To fix notation, let $E = N_C[\posmodalong R]$ and 
$$E_{n,m} = E[x_1 \biposmod  y_1][x_1 \posmod  2y_1][y_1 \posmod 2x_1]\cdots [x_n \biposmod  y_n][x_n \posmod  2y_n][y_n \posmod 2x_n][z_1 \biposmod  w_1]\ldots [z_m \biposmod  w_m].$$

By Remark \ref{rem: proving stability}, to show $N_D$ is stable it suffices to prove that subbundles of $E_{n,m}$ satisfy the slope bounds required by Remark \ref{rem: proving stability} for $n + m = 3$.  Note that $\mu(E_{n,m}) = 10 + 2n + \frac{2}{3}m$.  We analyze these bundles according to the value of $n$.

First we handle the case $n = 3$.  Observe that \(N_C\) is stable of slope \(8 \frac{1}{3}\), while $E$ is a $5$-cyclic modification of $N_C$ of slope \(10\).  Thus, by Lemma \ref{lem:black magic} it is semistable.  Specializing $y_1$ to $x_2$, and $y_2$ to $x_3$, and $y_3$ to $x_1$ induces a specialization of \(E_{3,0}\) to \(E(2x_1 + 2x_2 + 2x_3)\), so \(E_{3, 0}\) is semistable as well.  Thus the slope bounds of Remark \ref{rem: proving stability} are trivially satisfied.

It remains to consider the cases \(n = 0, 1, 2\). For the remainder of the proof, we use Lemma \ref{lem: onion degeneration} to specialize \(R\) to a union of three bisecant lines $\overline{pq_i}$ through $p, q_i \in C$.
    
When $n = 0$, we additionally specialize \(z_1\) to \(p\). The resulting specialization \(E_{0, 3}'\) of \(E_{0, 3}\) fits into an exact sequence
$$0 \to N_{C\to p}(q_1 +  q_2 + q_3 + p + w_1) \to E_{0,3}' \to N_{\overline{C}}(2p)[z_2 \biposmod w_2][z_3 \biposmod w_3] \to 0.$$
The subbundle has slope 12 and the quotient is semistable by Lemma \ref{lem: deg 4 elliptic curve modifications}.  This proves our claim for $n = 0$.
    
Similarly, when $n = 1$, we specialize $x_1$ to $p$.  The resulting specialization \(E_{1,2}'\) of \(E_{1,2}\) fits into an exact sequence
$$0 \to N_{C\to p}(q_1 + q_2 + q_3 + p + 2y_1) \to E_{1,2}' \to N_{\overline{C}}(3p)[y_1 \posmod p][z_1 \biposmod w_1][z_2 \biposmod w_2] \to 0.$$
The subbundle has slope 13 and the quotient is stable by Lemma \ref{lem: deg 4 elliptic curve modifications}.  This proves our claim for $n = 1$.

Lastly, suppose $n = 2$.  Specialize $x_1$ and $z_1$ to $p$.  The resulting specialization $E_{2,1}'$ of $E_{2,1}$ fits into an exact sequence
$$0 \rightarrow N_{C\rightarrow p}(q_1 + q_2 + q_3 + 2p + 2y_1 + w_1) \rightarrow E_{2,1}' \rightarrow N_{\overline{C}}(3p + x_2 + y_2)[y_1 \posmod p][x_2 \biposmod y_2] \rightarrow 0.$$
The subbundle has slope $15$ while the quotient is a direct sum of two line bundles of degree $14$ and $15$.  This proves our claim for $n = 2$.  We conclude that $N_D$ is stable.
\end{proof}

\begin{proof}[Proof of Theorem \ref{thm-mainstab}]  The cases that are not stable were discussed at length in \S \ref{sec-unstablecases}. Thus, in order to prove Theorem \ref{thm-mainstab} we need to verify the stability of the finitely many $(g,d)$ listed in Theorem \ref{thm: basecases}. Corollary \ref{cor: most cases} and Lemmas \ref{lem: extend to genus 8}--\ref{lem: last lem} verify these cases. This concludes the proof.
\end{proof}

Finally we consider a general BN-curve $D \subset \PP^4$ of genus $7$ and degree $10$. We address the stability of the normal bundle in the unique quadric $Q$ containing $D$, as promised in Section~\ref{ss:710}.

\begin{lemma}\label{lem: stability in quadric}
Let $D \subset \PP^4$ be a general BN-curve of genus $7$ and degree $10$. Let $Q$ be the (unique) smooth quadric containing $D$. Then \(N_{D/Q}\) is stable.
\end{lemma}

\begin{proof}
Let $C \subset \PP^4$ be a general elliptic curve of degree 5 and $R$ be a five-secant twisted cubic to $C$.  To prove $N_{D/Q}$ is stable, we will show that $N_D$ does not contain any line subbundles of degree $21$.  We will do so by specializing $D$ to the nodal union of $C$ with $R$ and two general bisecant lines to $C$.  Let $E = N_{C\cup R}|_C$.  By Remark \ref{rem:iterative extra juice}, it is sufficient to show that
\begin{itemize}
\item $E_{0,2}$ has no line subbundles of degree $13$ or greater;
\item $E_{1,1}$ has no line subbundles of degree $14$ or greater;
\item $E_{2,0}$ has no line subbundles of degree $15$ or greater.
\end{itemize}
We will prove these bounds by specializing $R$ and the bisecant lines.

First, consider $E_{0,2}$.  Specialize $R$ to a union of three bisecant lines $\overline{pq_i}$ through a common point $p \in C$.  Then specialize $z_1$ to $p$ as well.  The resulting specialization $E_{0,2}'$ of $E_{0,2}$ fits into an exact sequence
$$0 \to N_{C\to p}(p + q_1 + q_2 + q_3 + w_1) \to E_{0,2}' \to N_{\overline{C}}(2p)[z_2 \biposmod w_2] \to 0.$$
The subbundle has degree 12, and the quotient $N_{\overline{C}}(2p)[z_2 \biposmod w_2] \cong \OO_{\overline{C}}(2)(2p) \oplus \OO_{\overline{C}}(2)(2p+z_2 + w_2)$ is a direct sum of line bundles of degree 10 and 12, respectively.  This proves $E_{0,2}$ has no line subbundles of degree $13$ or greater.

Next, consider $E_{1,1}$.   Specialize $R$ to a union of three bisecant lines $\overline{pq_i}$ through a common point $p \in C$.  Then specialize $x_1$ to $p$ as well.  The resulting specialization $E_{1,1}'$ of $E_{1,1}$ fits into an exact sequence
$$0 \to N_{C\to p}(p + q_1 + q_2 + q_3 + 2y_1) \to E_{1,1}' \to N_{\overline{C}}(3p)[z_1 \biposmod w_1][y_1 \posmod p] \to 0.$$
The subbundle has degree 13, while $N_{\overline{C}}(3p)[z_1 \biposmod w_1][y_1 \posmod p] \cong \OO_{\overline{C}}(2)(3p + y_1) \oplus \OO_{\overline{C}}(2)(3p+z_2 + w_2)$ is a direct sum of line bundles of degree 12 and 13, respectively.  This proves $E_{1,1}$ has no line subbundles of degree $14$ or greater.

Lastly, consider $E_{2,0}$.  Let $p, q_1, \ldots , q_4 \in C$ be the points in a general hyperplane section of $C$.  Specialize $R$ to the union of a bisecant line $\overline{pq_1}$ and a conic $T$ passing through $\overline{pq_1}$ that meets $C$ at the three points $q_2, q_3, q_4$.  Specialize $x_1$ and then $x_2$ to $p$.  The resulting specialization $E_{2,0}'$ of $E_{2,0}$ fits into an exact sequence
$$0 \to N_{C\to p}(2p + 2y_1 + 2y_2 + q_1) \to E_{2,0}' \to N_{\overline{C}}(3p)[y_1 + y_2 \posmod p][q_2 + q_3 + q_4 \posmodalong \overline{T}][p \posmod y_2] \to 0$$
where $\overline{T} \subset \PP^3$ is the image of $T$ under projection from $p$.  
The subbundle has degree $14$ so it suffices to prove that the quotient has no line subbundles of degree $14$ or greater. For this we further specialize \(y_1\) to \(q_2\), and then \(\overline{T}\) to the union of lines \(\overline{q_1 q_2}\) and \(\overline{q_3 q_4}\).
This induces a specialization of the quotient to \(N_{\overline{C}}(3p + q_2)[y_2 \biposmod p][q_3 \biposmod q_4]\), which is semistable by Lemma \ref{lem: deg 4 elliptic curve modifications} as desired.
\end{proof}

\end{document}